\documentclass{elsarticle}
\makeatletter
\def\ps@pprintTitle{%
 \let\@oddhead\@empty
 \let\@evenhead\@empty
 \def\@oddfoot{\centerline{\thepage}}%
 \let\@evenfoot\@oddfoot}
\makeatother

\usepackage{hyperref}
\usepackage{times}
\usepackage{amssymb}
\usepackage{epsfig}
\usepackage{graphicx}
\usepackage{amsmath}
\usepackage{amsthm}
\usepackage{xypic}

\usepackage{bbm}
\usepackage{multirow}
\allowdisplaybreaks

\newtheorem{defi}{Definition}
\newtheorem{thm}{Theorem}
\newtheorem{proposition}{Proposition}

\newtheorem{lemma}{Lemma}
\newtheorem{remark}{Remark}
\usepackage
{todonotes}

\begin{document}

\begin{frontmatter}

\title{Comparing Curves in Homogeneous Spaces}

\author{Zhe Su\quad Eric Klassen\quad Martin Bauer}
\address{Florida State University, Department of Mathematics, Tallahassee, FL}
\ead{zsu@math.fsu.edu, klassen@math.fsu.edu, bauer@math.fsu.edu}

\begin{abstract}
Of concern is the study of the space of curves in homogeneous spaces. Motivated by applications in shape analysis we identify two curves if they only 
differ by their parametrization and/or a rigid motion. For curves in Euclidean space the Square-Root-Velocity-Function (SRVF) allows to define and efficiently compute a distance on this infinite dimensional quotient space. In this article we present a generalization of the SRVF to curves in homogeneous spaces.
We prove that, under mild conditions on the curves, there always exist optimal reparametrizations realizing the quotient distance and demonstrate the efficiency of our framework in selected numerical examples.
\end{abstract}


\end{frontmatter}


\section{Introduction}

Comparing shapes of curves is a topic of intrinsic interest and, in addition, it is of relevance in many applications in the broad area of shape analysis \cite{Younes2010,SrKl2016,BaBrMi2014}. Usually the notion of ``shape" means comparing curves without regard to rigid motions or reparametrizations. Thus, it implies modding out the space of parametrized curves by the group of rigid motions, and/or the group of reparametrizations. We might be interested in curves in a flat Euclidean space (for example, the outline of an image in a photograph), or we might be interested in curves that lie on a space that is itself curved (for example, hurricane tracks on the surface of the earth or paths of positive definite symmetric matrices in brain connectivity analysis). This paper is primarily concerned with the second of these two cases.

To outline our approach to this problem, let $\mathcal{P}([0,1], M)$ denote the set of para\-metrized curves in a Riemannian manifold $M$. Thinking of $\mathcal{P}([0,1], M)$ as an infinite dimensional manifold, we wish to equip it with a Riemannian metric that is invariant under the group of rigid motions of $M$, and under the group of reparametrizations. In this way, we can induce a metric on the quotient of $\mathcal{P}([0,1], M)$ by either, or both, of these groups. This will allow us to quantify difference between  shapes of curves by calculating the length of the shortest geodesic joining them in the quotient space. We can also perform statistical analyses on sets of curves by using  techniques of non-linear statistics on this quotient manifold. 

For the case $M=\mathbb R^n$ several metrics have been defined satisfying the required invariances, see, e.g., \cite{BaBrMi2014,MiMu2007,SrKlJoJe2011,Sh2008,YeMe2005,MeYeSu2008} and the references therein. The main goal of this paper is to take a particularly useful one of these metrics, the elastic metric associated with the ``square root velocity function" (SRVF), and generalize it to curves in a homogeneous manifold $M$. (A homogeneous manifold is a quotient of a Lie group by a compact subgroup.) \\

\noindent
{\bf Previous work on curves in $\mathbb R^n$}
In \cite{MiMu2006,MiMu2005,BaBrHaMi2012}, Michor and Mumford showed that the simplest reparametrization invariant $L^2$-metric on $\mathcal{P}([0,1], \mathbb R^n)$ is an inadequate choice for shape analysis as it results in vanishing geodesic distance, i.e., for any two curves $c_1,c_2\in\mathcal{P}([0,1], \mathbb R^n)$ there exist paths of arbitrarily short length connecting them. Subsequently it has been shown in \cite{MiMu2007} that this degeneracy can be overcome by adding higher order derivatives in the definition of the metric, yielding to the class of reparametrization invariant Sobolev metrics. While this class of metrics allows one to prove strong theoretical results \cite{BrMiMu2014}, it can be difficult to calculate the corresponding minimizing geodesics and thus obtain the distance function on the shape space of curves. (See also the recent article on  a numerical framework for general second order Sobolev metrics \cite{BauerAppl2}.)

For planar curves (i.e., $M=\mathbb R^2$), Younes et al. \cite{YoMiShMu2008,Yo1998} consider a specific first order Sobolev metric, that gives rise to an efficient method for calculating geodesics in the space of parametrized curves. Their methods are, however, very specific to $\mathbb R^2$. 

In \cite{MiSrJo2007}, Mio et al. considered a family of ``elastic metrics" on the space of planar curves. Intuitively, this family allows one to attach different weights to perturbations in the tangent direction (``stretching") and in the normal direction (``bending"). A precise formula for this metric is given by 
\begin{equation}
G_c(v_1, v_2)=\int_0^1a^2\langle D_sv_1^\bot,D_sv_2^\bot\rangle+b^2\langle D_sv_1^\top,D_sv_2^\top\rangle ds,
\end{equation}
where $c:[0,1]\to\mathbb R^2$ is a parametrized curve, $v_1$ and $v_2$ are vector fields along this curve, $D_s$ and $ds$ denote differentiation and integration with respect to arc-length, and $D_sv_1^\bot$ (resp. $D_sv_1^\top$) denotes the component of $D_sv_1$ that is normal (resp. tangent) to the tangent vector $c'$ of the curve. For the case $a=b$, this metric is precisely the one studied by Younes et al. 

In \cite{SrKlJoJe2011}, Srivastava et al. found, analogous to the transformation of \cite{YoMiShMu2008}, an efficient representation of the elastic metric with parameter values $a=1$ and $b=\frac12$. In contrast to the work  \cite{YoMiShMu2008}
their framework is valid for curves with values in arbitrary $\mathbb R^n$. This method, known as the square root velocity function, has proved extremely successful for computations and has been used in numerous applications in shape analysis, see \cite{SrKl2016} and the references therein. The SRVF method has several important properties:
\begin{enumerate}
\item The metric is extended to the space of all absolutely continuous curves, a much larger space of curves than smooth immersions.
\item The space of open parametrized curves is metrically and geodesically complete, and there are explicit formulas to compute geodesics.
\item As a consequence of 2, modding out by the reparametrization group can be implemented efficiently using, e.g., a dynamic programming algorithm. 
\item Geodesics (in the sense of metric spaces) have been shown to exist in the space of unparametrized curves, if one places mild restrictions on the curves, that is if both curves are $C^1$, see \cite{Bru2016}, or if at least one of the two curves is piecewise linear, see \cite{LaRoKl2015}.
\end{enumerate} 
Recently a generalization of the SRVF for a larger range of the parameters $a$ and $b$ has been introduced in \cite{BaBrMaMi2014}. \\

\noindent
{\bf Previous work on extending the SRV-framework to general manifolds:}
Because of the efficiency of the SRVF in analyzing curves in $\mathbb R^n$, several papers have been written generalizing this framework to curves in general Riemannian manifolds \cite{SuKuKlSr2014, ZhSuKlLeSr2015,LeArBa2015}, in Lie groups \cite{CeEsSch2015} and in homogeneous spaces \cite{CeEiEsSc2017,CeEiEsSc2017a,SuKlBa2017a}. In \cite{SuKuKlSr2014}, all tangent vectors of curves are parallel transported along minimal geodesics to the tangent space at a fixed reference point in the manifold. This method is computational effective, but it introduces distortions for curves that venture far away from the reference point and, as a result, the metric is not invariant under the group of isometries of the underlying manifold. Methods in \cite{ZhSuKlLeSr2015,LeArBa2015} are different adaptations of the SRVF for curves with values in manifolds. These methods avoid the arbitrariness and distortion resulting from the choice of a reference point and they are invariant under the isometries of the manifold; however, they have great computational costs. \\

\noindent
{\bf Contributions of this paper:}
In this paper, we generalize the SRVF to curves with values in a homogeneous space $M=G/K$, where $G$ is a Lie group and $K$ is a compact Lie subgroup. 
Our metric is both computationally efficient
and  invariant under the isometry group of the manifold $M$. In particular it avoids the distortion and arbitrariness  of the reference point in \cite{SuKuKlSr2014}. 
Independently of the present work, Celledoni et al. \cite{CeEiEsSc2017,CeEiEsSc2017a} defined a framework for comparing curves in a homogeneous space that is similar to the method given in the present paper, inasmuch as it extends the definition of the SRVF using the Lie group structure of $G$. However, their method applies only to sets of curves that all start at the same point in $M$, which is a severe limitation for applications in shape analysis. In the present work, we use a topological twisting construction to define the SRVF on the space of all absolutely continuous curves in a homogeneous space, not just those starting at a specific point. Compared to previous attempts, our approach has the advantage that it yields explicit formulas for geodesics and geodesic distance, which makes the matching between curves computationally efficient. 

While the class of homogeneous spaces is a very restricted class of manifolds,  it should be noted that the manifolds arising in applications very often fall into this class. Examples of homogeneous spaces include Euclidean spaces, spheres, hyperbolic spaces, Grassmannians, spaces of positive definite symmetric matrices, as well as all Lie groups. 

Our approach is based on first defining the SRVF for curves with values in Lie groups and then putting a metric on the space of parametrized curves in $M$ based on their horizontal lifts to curves in the Lie group $G$. We left translate (instead of parallel transporting) the tangent vectors of these curves to the Lie algebra of $G$. This avoids the distortion resulting from the choice of reference points.
Furthermore, assuming mild conditions on the curves, we prove the existence of optimal reparametrizations both for curves in Lie groups and in homogeneous spaces, thereby generalizing the corresponding results 
\cite{Bru2016,LaRoKl2015} for curves in $\mathbb R^n$. We then present the implementation of our method and show selected examples demonstrating both the
effectiveness of our method and the influence of the curved ambient space.  

The method presented in this paper was originally introduced in an earlier conference paper  by the same authors \cite{SuKlBa2017a}, along with some implementation results. The current paper expands on this earlier paper, giving proofs (which were mostly omitted in the conference paper) of the central theoretical results underlying the method, as well as giving more illustrative implementation results.

%
%

\section{The SRVF for the space of curves with values in a Lie group}

In this section, we will focus on the space of absolutely continuous curves with values in a finite dimensional Lie group. We first recall the definition of an absolutely continuous curve in a smooth manifold:
\begin{defi}
	Let $(N,\mathcal{K})$ be a finite dimensional smooth manifold. A curve $\beta: [0,1]\to N$ is called \textbf{absolutely continuous} if for every local chart $(U, \phi)$ and every closed subset $[a, b]\subset \beta^{-1}(U)$, $\phi\circ\beta: [a,b]\to \mathbb R^N$ is absolutely continuous.
\end{defi}
A function $f:[0,1]\to\mathbb R^N$ is absolutely continuous if and only if $f$ has a derivative a.e., the derivative $f’$ is Lebesgue integrable and $f(t)-f(0) = \int_0^tf'(u)du$, see \cite{Fo1999}. By definition, it is easy to see that $\beta'$ exists a.e. on $N$ and that the length $\operatorname{L}(\beta) = \int_I\|\beta'\|_{\mathcal{K}}dt$ is well defined and finite. In the following let $I=[0,1]$. We denote the space of all absolutely continuous curves with values in $N$ by $AC(I, N)$. In this article N will be either a Lie group or a homogeneous space. For more information on absolutely continuous curves with values in manifolds we refer to the recent article \cite{Schmeding2016} and the references therein.

\subsection{Parametrized curves with values in a Lie group}
Let $G$ be a finite dimensional Lie group. We assume that $G$ is equipped with a left invariant Riemannian metric $\langle\cdot,\cdot \rangle^G$. Denote by $\mathfrak{g} =T_eG$ the Lie algebra of $G$. Following the square root velocity framework (SRVF) introduced for curves in $\mathbb R^N$ by Srivastava et al. in \cite{SrKlJoJe2011}, we define the map 
\begin{align}\label{eq.Q.map}
\begin{cases}
Q: AC(I, G)\to G\times L^2(I,\mathfrak{g})\\
Q(\alpha)=(\alpha(0), q_{\alpha}),
\end{cases}
\end{align}
where 
\begin{align}\label{eq.q.map}
q_{\alpha}(t)=\left\{  \begin{array}{lcr}
\dfrac{L_{\alpha(t)^{-1}}\alpha'(t)}{\sqrt{\|\alpha'(t)\|}} &\alpha'(t)\neq 0\\
0       &\alpha'(t)=0
\end{array} \right.
\end{align}
In this definition, we use $L_{\alpha(t)^{-1}}$ to denote the left translation applied to elements of $G$, and also to tangent vectors. The norm $\|\cdot\|$ is induced by the left invariant metric $\langle\cdot,\cdot \rangle^G$ on $G$ and $\alpha(t)^{-1}$ is the inverse element of $\alpha(t)$ in $G$.
Note that $\alpha\mapsto q_\alpha$ is a map from $AC(I, G)\to L^2(I, \mathfrak{g})$. In most of this paper, without causing confusion, we will simply write $q$ instead of $q_{\alpha}$.
It is easy to see that this map $Q$ is well defined and we have the following proposition. 
\begin{proposition}\label{pro.QBijection}
	The map $Q:AC(I, G)\to G\times L^2(I,\mathfrak{g})$ is a bijection.
\end{proposition}
\begin{proof}
	Given $(\alpha_0, q)\in G\times L^2(I,\mathfrak{g})$, the preimage $\alpha$ under $Q$ is a solution of the following initial value problem:
	\begin{equation}\label{ODE}
	\begin{cases}
	\alpha'=L_{\alpha}(\|q(t)\|q(t))\\
	\alpha(0)=\alpha_0.
	\end{cases}
	\end{equation}
	In the case of $G=\mathfrak{g}=\mathbb{R}^N$, the existence and uniqueness of such a solution was proven by Robinson in \cite{Rob2012}. In the case of any finite dimensional Lie group $G$, let $(\alpha_0, q)\in G\times L^2([0,1],\mathfrak{g})$. Then $\|q(t)\|q(t)\in L^1(I, \mathfrak{g})$. By \cite[Theorem~C]{Glo2015}, $G$ is $L^1$-regular, which means that there is a unique absolutely continuous curve $\eta$ with values in $G$ such that
	\begin{equation}
	\begin{cases}
	\eta'(t) = L_{\eta(t)}(\|q(t)\|q(t))\\
	\eta(0) = e.
	\end{cases}
	\end{equation}
	This result is a special case of a corresponding result for infinite dimensional Lie groups.
	Let $\alpha(t) = \alpha_0\eta(t)$. Then $\alpha\in AC(I, G)$ and it is the unique solution of the initial value problem \eqref{ODE}. Therefore, $Q$ is a bijection.
\end{proof}
This Proposition also can be proved directly by considering local charts and using Carath\'{e}odory's existence theorem \cite[Theorem ~5.1]{Ha1980} and uniqueness theorem \cite[Theorem ~5.3]{Ha1980}.

Note that $G\times L^2(I,\mathfrak{g})$ is a smooth manifold and it has a natural product metric given by
\begin{equation}\label{eq.metric.ProductSpace}
	\langle(x_1, v_1), (x_2, v_2)\rangle^P_{(y, u)} = \langle x_1, x_2\rangle^G +\int_I\langle v_1, v_2\rangle^Gdt
\end{equation}
where $(y, u)\in G\times L^2(I,\mathfrak{g})$ and $(x_1, v_1), (x_2, v_2)\in T_{(y,u)}\left(G\times L^2(I, \mathfrak{g})\right)$.

\begin{remark}
	Since $Q$ is a bijection, $AC(I,G)$ can be equipped with a smooth structure such that $Q$ is a diffeomorphism. We can then consider the Riemannian metric on $AC(I, G)$ obtained by pulling back the metric from $AC(I,G)$ using $Q$.  However, it is worth noting that if $AC(I, G)$ is equipped with its ``standard" smooth structure, then $Q$ is not differentiable at any curve $\alpha$ with $\alpha'=0$ on a set of positive measure, and thus does not induce a smooth Riemannian metric on $AC(I,G)$; see \cite{Bru2016} for more details.
\end{remark}

Let $\alpha_1, \alpha_2\in AC(I, G)$,	$Q(\alpha_1)=(\alpha_1(0),q_1)$ and	$Q(\alpha_2)=(\alpha_2(0),q_2)$.
The distance function on $AC(I,G)$ is of the form:
\begin{align}\label{eq.disACG}
d(\alpha_1, \alpha_2)=\sqrt{d_{G}^2(\alpha_1(0), \alpha_2(0))+\|q_1-q_2\|^2_{L^2}},
\end{align}
where $d_{G}$ is the geodesic distance on $G$, and $\|\cdot\|_{L^2}$ refers to the $L^2$ norm. The right hand side of this equation is the geodesic distance of $Q(\alpha_1)$ and $Q(\alpha_2)$ on the product space $G\times L^2(I, G)$.

Consider the monoid $\tilde{\Gamma}$ of reparametrizations, where 
\begin{equation}
	\tilde{\Gamma} = \{\gamma: I\to I, \gamma\ \text{is abs. cont.}, \gamma(0)= 0,\gamma(1)= 1, \gamma'\geq 0\ a.e.\}.
\end{equation}
This monoid $\tilde{\Gamma}$ is the closure of the reparametrization group $\Gamma$ in $AC(I, \mathbb R)$, where
\begin{equation}
	\Gamma = \{\gamma: I\to I, \gamma\ \text{is abs. cont.}, \gamma(0)= 0,\gamma(1)= 1, \gamma'> 0\ a.e.\}.
\end{equation}
(The closure is with respect to the SRVF metric; see, e.g, \cite{SrKl2016}).
The semigroup $\tilde{\Gamma}$ acts on $AC(I, G)$ by right composition. We can consider in addition the action of $G$ on $AC(I, G)$ by left multiplication. Given $g\in G,\,\gamma\in\Gamma$, the corresponding actions of $G$ and $\Gamma$ on the product space $G\times L^2(I,\mathfrak{g})$ are as follows:
\begin{align}
g\bullet(\alpha_0, q) &= (g\alpha_0, q)\\
(\alpha_0, q)\star\gamma &=\left(\alpha_0,\  q\circ\gamma\sqrt{\gamma'}\right),
\end{align}
where $(\alpha_0, q)\in G\times L^2(I, \mathfrak{g})$. It is clear that these two actions commute with each other.
We have the following proposition.
\begin{proposition}
	The distance function \eqref{eq.disACG} on $AC(I,G)$ is invariant under the actions of $G$ and under the action of $\tilde{\Gamma}$. 
\end{proposition}
\begin{proof}
	Since the metric on $G$ is left invariant, it is easy to see that $G$ acts by isometries. Let $\alpha_1, \alpha_2\in AC(I, G),\ \gamma\in\tilde{\Gamma}$, we have
	\begin{align}
		d(\alpha_1\circ\gamma, \alpha_2\circ\gamma) &= \sqrt{d_{G}^2(\alpha_1(0), \alpha_2(0))+\|q_1\star\gamma-q_2\star\gamma\|_{L^2}^2}\notag\\
		& = \sqrt{d_{G}^2(\alpha_1(0), \alpha_2(0))+\|(q_1\circ\gamma-q_2\circ\gamma)\sqrt{\gamma'}\|_{L^2}^2}\notag\\
		& = \sqrt{d_{G}^2(\alpha_1(0), \alpha_2(0))+\|q_1-q_2\|^2_{L^2}}\notag\\
		& = d(\alpha_1, \alpha_2).
	\end{align}
	Thus $\tilde{\Gamma}$ also acts on $AC(I, G)$ by isometries.
\end{proof}

We now give an interpretation of the metric on $AC(I, G)$. Let 
\begin{equation}
	\delta^l: AC(I, G)\to L^1(I, \mathfrak{g}),\quad \delta^l(\alpha) = \alpha^{-1}\alpha',
\end{equation}
which is called the {\it left logarithmic derivative}, see \cite{Mi2008}.
\begin{proposition}
    Let $\alpha$ be an absolutely continuous curve with values in $G$ that has non-vanishing derivatives a.e. and let $u, v\in T_{\alpha}AC(I, G)$. The pullback metric $\mathcal{G}$ on $AC(I, G)$ at $\alpha$ is given by
    \begin{align}\label{eq.metric.AC.G.G}
    \mathcal{G}_{\alpha}(u, v)=\langle u(0),v(0)\rangle^G
    +\int \langle D_s u^N, D_s v^N\rangle^G+\frac14\langle D_s u^T, D_sv^T\rangle^Gds,
    \end{align}
    where $D_s(u)=\frac{1}{\|\alpha'\|}\delta^l_{*\alpha}(u)$, $D_su^T=\langle D_su,
    \frac{\delta^l(\alpha)}{\|\alpha'\|}\rangle^G\left(\frac{\delta^l(\alpha)}{\|\alpha'\|}\right)$, $D_su^N=D_su-D_su^T$ and we integrate with respect to the arclength $ds=\|\alpha'(t)\|dt$.  
\end{proposition}
\begin{proof}
	First, we have the differential of $Q$ at $\alpha$:
	\begin{equation}
		\begin{cases}
		Q_{*\alpha}: T_{\alpha}AC([0,1], G)\to T_{(\alpha(0),q)}(G\times L^2([0,1],\mathfrak{g}))\notag\\
		Q_{*\alpha}u=\left(u(0),q_{*\alpha}u\right),
		\end{cases}
	\end{equation}
	where we use $q$ to denote the map from $AC(I, G)\to L^2(I, \mathfrak{g})$ defined in \eqref{eq.q.map}. We have the differential $q_{*\alpha}: T_{\alpha}AC([0,1], G)\to T_qL^2([0,1],\mathfrak{g})$ and
	\begin{align}
	q_{*\alpha}u=&\|\alpha'\|^{1/2}D_s(u)-\frac{1}{2} \| \alpha'\|^{-3/2} \langle D_s u, \delta^l(\alpha)\rangle^G \delta^l(\alpha).
	\end{align}
	A computation of this can be found in \cite{CeEsSch2015}. Note that the pullback metric $\mathcal{G}$ on $AC(I, G)$ at $\alpha$ is defined by
	\begin{align}
	\mathcal{G}_{\alpha}(u, v)&=\langle Q_{*\alpha}u, Q_{*\alpha}v\rangle^P_{Q(\alpha)}\notag\\
	&=\langle u(0), v(0)\rangle^G + \int_I\langle q_{*\alpha}u, q_{*\alpha}v\rangle^Gdt.
	\end{align}
	By direct computation, the expression \eqref{eq.metric.AC.G.G} of $\mathcal{G}$ follows immediately.
\end{proof}
In the case of $G =\mathbb{R}^N$, the last two terms in the formula of the metric $\mathcal{G}$ will become the elastic metric as defined in \cite{MiSrJo2007}. On Lie groups, if using right trivialization instead of left, the last two terms form the pullback metric obtained by Celledoni et al. in \cite{CeEsSch2015}. However, it is different from the metrics introduced by Le Brigant et al. \cite{LeArBa2015} and Zhang et al. \cite{ZhSuKlLeSr2015} for arbitrary Riemannian manifolds, since their methods are based on parallel transport, which is different from the left translation on Lie groups used in our method.

\subsection{Unparametrized curves with values in a Lie group}

Since we are interested in unparametrized curves, we now define an equivalence relation on $AC(I, G)$ as follows:
given $\alpha_1$ and $\alpha_2$, we say $\alpha_1\sim\alpha_2$ if and only if there exist $\bar{\alpha}\in AC(I, G)$, $\gamma_1, \gamma_2 \in\tilde{\Gamma}$ such that 
$\alpha_1=\bar{\alpha}\circ\gamma_1$ and $\alpha_2=\bar{\alpha}\circ\gamma_2$. This is equivalent to $\overline{\alpha_1\Gamma} = \overline{\alpha_2\Gamma}$ with the metric defined on $AC(I, G)$, see \cite{LaRoKl2015,Bru2016}. Then we define the shape space $\mathcal{S}(I, G)$ as the set of equivalence classes under $\sim$, that is, the quotient space of parametrized curves with values in $G$
\begin{equation}
	\mathcal{S}(I,G) = AC(I, G)/\sim.
\end{equation}
By the {\it shape} of a curve with values in $G$, we mean its equivalence class in $\mathcal{S}(I, G)$. The space $\mathcal{S}(I, G)$ is not a manifold but we can endow $\mathcal{S}(I, G)$ with a distance function so that it becomes a metric space: since the distance on $AC(I, G)$ is reparametrization invariant, we can consider the induced quotient distance on $\mathcal{S}(I, G)$
\begin{align}\label{eq.disSG}
	d([\alpha_1], [\alpha_2]) &= \inf_{\gamma_1,\gamma_2\in\tilde{\Gamma}}d(\alpha_1\circ\gamma_1, \alpha_2\circ\gamma_2)\notag\\
	& = \inf_{\gamma_1,\gamma_2\in\tilde{\Gamma}}\sqrt{d_G^2(\alpha_1(0),\alpha_2(0))+ \|q_1\star\gamma_1-q_2\star\gamma_2\|_{L^2}^2}.
\end{align}
We now focus on the following important question: given absolutely continuous curves $\alpha_1$ and $\alpha_2$ with values in $G$, do there always exist $\gamma_1,\gamma_2\in\tilde{\Gamma}$ realizing this infimum? For curves in $\mathbb R^N$ this question has been answered in \cite{Bru2016}. We will show that the analogous results are also true for curves in Lie groups. Following \cite{Bru2016} we immediately obtain the result: 
\begin{thm}\label{CounterExampleLieGroup}
	If $\dim G\geq 2$, there exists a pair of Lipschitz curves with values in $G$ such that the infimum \eqref{eq.disSG} cannot be obtained by any pair of reparametrizations.
\end{thm}
\begin{proof}
	We follow the idea in \cite{Bru2016} with small adjustments. Set
	$v_1 = (\cos\epsilon t)w_1+ (\sin\epsilon t)w_2$,
	$v_2 = -\frac12w_1+\frac{\sqrt{3}}{2}w_2$, and
	$v_3 = -\frac12w_1-\frac{\sqrt{3}}{2}w_2$,
	where $0<\epsilon\leq\frac 16$ and $\{w_1, w_2\}$ is an orthonormal pair of elements of $\mathfrak{g}$. Let $B\subset I$ be a closed and nowhere dense set with Lebesgue measure $\frac{1}{2}$ and $A = I\backslash B$. We define
	\begin{align}
		&p_1(t) = v_1(t)\mathbbm{1}_A(t)+v_2(t)\mathbbm{1}_B(t)\notag\\
		&p_2(t) = v_1(t)\mathbbm{1}_A(t)+v_3(t)\mathbbm{1}_B(t);
	\end{align}
	then the preimages of $p_1(t)$ and $p_2(t)$ under the map $Q$ are Lipschitz curves with values in $G$, for which no optimal reparametrizations exist. See more details in \cite{Bru2016}.
\end{proof}

However, under slightly more restrictive conditions we get existence of optimal reparametrizations. The existence of optimal reparametrizations for piecewise linear curves in $\mathbb{R}^N$ has been proven by Lahiri et al. \cite{LaRoKl2015} in 2015: they show the existence of optimal reparametrizations if at least one of the curves is piecewise linear (PL). For $C^1$ curves in $\mathbb{R}^N$, the existence of optimal reparametrizations was shown by Bruveris \cite{Bru2016} in 2016. In the following we generalize these results for curves with values in a Lie group G. First, we first extend the definition of PL curves to Lie Group valued curves. 
\begin{defi}
	We call a curve $\alpha$ with values in a Lie group $G$ a \textbf{generalized PL curve} if there exists a sequence $0= t_0\leq t_1\leq t_2 \cdots t_{k-1}\leq t_{k}= 1$ such that
	\begin{align}\label{PLLieGroup}
	\alpha(t)=\left\{  \begin{array}{lcr}
	\alpha_0\exp(v_1t) & t\in [0, t_1]\\
	\alpha_0\exp(v_1t_1)\exp(v_2(t-t_1))      & t\in [t_1, t_2]\\
	\alpha_0\exp(v_1t_1)\exp(v_2(t_2-t_1))\exp(v_3(t-t_2)) & t\in [t_2, t_3]\\
	\vdots\\
	\alpha_0\exp(v_1t_1)\cdots \exp(v_k(t-t_{k-1})) & t\in [t_{k-1}, 1]
	\end{array} \right.
	\end{align}
	where $\alpha_0\in G$, $v_j\in\mathfrak{g}, 1\leq j\leq k$, and where $\exp: \mathfrak{g}\to G$ denotes the Lie group exponential.
\end{defi}
We are now able to formulate our main result:.

\begin{thm}\label{OpExstLieGroupMain}
	Let $G$ be a Lie group and $\alpha_1, \alpha_2\in AC(I, G)$. Assume in addition that one of the following two conditions is satisfied:
	\begin{enumerate}[a)]
		\item $\alpha_1$ or $\alpha_2$ is a generalized PL curve;
		\item $\alpha_1, \alpha_2\in C^1(I,G)$ with $\alpha_1',  \alpha_2' \neq 0$ a.e.
	\end{enumerate}
	Then there exists a pair of reparametrizations $\gamma_1, \gamma_2\in\tilde{\Gamma}$ such that
	\begin{equation}
	d(\alpha_1\circ\gamma_1, \alpha_2\circ\gamma_2) = d([\alpha_1], [\alpha_2]).
	\end{equation}
	Furthermore, if $\alpha_1$ and $\alpha_2$ are both generalized PL curves, then this distance can be realized by a pair of piecewise linear functions in $\tilde{\Gamma}$.
\end{thm}

In preparation for the proof of Theorem \ref{OpExstLieGroupMain}, we give the definition of a step map in $L^2(I,\mathfrak{g})$.

\begin{defi}
	We call $q\in L^2(I,\mathfrak{g})$ a \textbf{step map} if there exists a finite sequence $0= t_0\leq t_1\leq t_2 \cdots t_{k-1}\leq t_{k}= 1$ such that q is constant on each interval $(t_{j-1},t_{j})$, i.e., 
	$$q(t)=q_j, \quad \forall t\in (t_{j-1},t_{j}),$$
	where $q_j\in\mathfrak{g}$ for all $1\leq j\leq k$.
\end{defi}
%
%
The mapping $Q$ is defined using the first derivative of the curve $\alpha$. Thus in the case of curves with values in $\mathbb{R}^n$, there exists a one-to-one correspondence between PL-curves and step functions, see \cite{LaRoKl2015}. In our more intricate situation we have the following result:
\begin{lemma}\label{PLLieGroupCorrespondence}
	Let $\alpha_0\in G$ and let $q\in L^2(I, \mathfrak{g})$. Then $q$ is a step map if and only if $Q^{-1}(\alpha_0, q)$ is a generalized PL curve. 
\end{lemma}
\begin{proof}
	Assume that $q(t) = q_j\in\mathfrak{g}$ for all $t\in(t_{j-1},t_{j})$, where $0= t_0\leq t_1\leq t_2 \cdots t_{k-1}\leq t_{k}= 1$. By direct computation, we have
	\begin{align}
	\alpha(t)=\left\{  \begin{array}{lcr}
	\alpha_0\exp(v_1t) & t\in [0, t_1]\\
	\alpha_0\exp(v_1t_1)\exp(v_2(t-t_1))       & t\in [t_1, t_2]\\
	\alpha_0\exp(v_1t_1)\exp(v_2(t_2-t_1))\exp(v_3(t-t_2)) & t\in [t_2, t_3]\\
	\vdots\\
	\alpha_0\exp(v_1t_1)\exp(v_2(t_2-t_1))\cdots \exp(v_k(t-t_{k-1})) & t\in [t_{k-1}, 1]
	\end{array} \right.
	\end{align}
	where $v_j=q_j\|q_j\|$, $1\leq j\leq k$. Conversely, let $\alpha$ be a curve of the form (\ref{PLLieGroup}), then $q(t) = \frac{v_j}{\sqrt{\|v_j\|}}, \forall t\in(t_{j-1}, t_{j})$, which is obviously a step map. 	
\end{proof}

Using that $\mathfrak{g}$ is isomorphic to some $\mathbb{R}^N$, we can apply \cite[Theorem~4, Theorem~5]{LaRoKl2015} and \cite[Proposition ~15]{Bru2016} to this case to get the following proposition.

\begin{proposition}\label{OptimalExstStepLieGroup}
	Let $q_1, q_2\in L^2(I, \mathfrak{g})$. Assume in addition that one of the following two conditions is satisfied:
	\begin{enumerate}[a)]
		\item $q_1$ or $q_2$ is a step map;
		\item $q_1$ and $q_2$ are continuous with $q_1, q_2\neq 0$ a.e.
	\end{enumerate}
	Then there exist $\gamma_1, \gamma_2 \in\tilde{\Gamma}$ such that
	\begin{equation}
	\|q_1\star\gamma_1-q_2\star\gamma_2\|_{L^2}= \inf_{\tilde{\gamma_1},\tilde{\gamma_2}\in\tilde{\Gamma}}\|q_1\star\tilde{\gamma_1}- q_2\star\tilde{\gamma_2}\|_{L^2}.
	\end{equation}
	Furthermore, if $q_1$ and $q_2$ are both step maps, then this infimum can be realized by a pair of piecewise linear functions in $\tilde{\Gamma}$.
\end{proposition}

We now give a proof the main theorem. 

\begin{proof}[Proof of Theorem \ref{OpExstLieGroupMain}]
	Since the reparametrization group $\tilde{\Gamma}$ does not change the starting points of the curves, we only need to consider the term $\inf_{\tilde{\gamma_1},\tilde{\gamma_2}\in\tilde{\Gamma}}\|q_1\star\tilde{\gamma_1}- q_2\star\tilde{\gamma_2}\|_{L^2}$. If $\alpha_1$ or $\alpha_2$ is a generalized PL curve with values in $G$, by Lemma \ref{PLLieGroupCorrespondence}, $q_1$ or $q_2$ is a step map; if $\alpha_1, \alpha_2\in C^1(I,G)$ with $\alpha_1',  \alpha_2' \neq 0$ almost everywhere, then $q_1 = q(\alpha_1)$ and $q_2=q(\alpha_2)$, $q_1, q_2\neq 0$ a.e. are continuous with values in $\mathfrak{g}$. The results then follow immediately from Proposition~\ref{OptimalExstStepLieGroup}.
\end{proof}

\section{The SRVF for the space of curves with values in a homogeneous space}

In the following, let $M=G/K$ be a homogeneous space, where $G$ is a finite dimensional Lie group and $K$ is a compact Lie subgroup of $G$. The aim of this section is to develop the SRVF framework for curves with values in homogeneous spaces. The basic idea behind our construction is to lift the paths in $M$ to paths in the Lie group $G$ and use the previous defined SRVF framework to compare these curves. From here on, we will assume that the metric on $G$ is not only invariant under left multiplication by $G$, but also under right multiplication by $K$. (It is easy to prove the existence of such a metric by averaging over $K$.) Then the metric on $G$ induces a Riemannian metric on $M$ that is invariant under the left action by $G$, see \cite{Petersen1998}.

\subsection{Parametrized curves with values in a homogeneous space}
Denote by $\mathfrak{k}$ the Lie algebra of $K$ and by $\mathfrak{k}^{\perp}$ the orthogonal complement of $\mathfrak{k}$ in $\mathfrak{g}$. Then $\mathfrak{g} = \mathfrak{k}\oplus\mathfrak{k}^{\perp}$. Let $AC^{\perp}(I, G)$ denote the set of absolutely continuous curves in $G$ which are orthogonal to each coset of $K$ they meet. Then we have the following lemma.
\begin{lemma}
	The map $Q$ restricts to a bijection $AC^{\perp}(I, G)\to G\times L^2(I,\mathfrak{k}^{\perp})$.
\end{lemma}
\begin{proof}
	Since the metric on $G$ is left invariant, $\alpha\in AC^{\perp}(I, G)$ if and only if $L_{\alpha^{-1}}\alpha'\in \mathfrak{k}^{\perp}$, which is equivalent to $q = \frac{L_{\alpha^{-1}}\alpha'}{\sqrt{\|\alpha'\|}}\in\mathfrak{k}^{\perp}$.
\end{proof}
Now let $K$ act on $G\times L^2(I, \mathfrak{k}^{\perp})$ from the right as follows: 
\begin{equation}\label{eq.Kaction}
	(\alpha_0, q)*y = (\alpha_0y, y^{-1}qy),
\end{equation}
where $y\in K, \alpha_0\in G$ and $q\in L^2(I,\mathfrak{k}^{\perp})$. Since the metric on $G$ is bi-invariant with respect to $K$, this action is by isometries. 
\begin{proposition}\label{pro.Bi.M}
	The map $Q$ induces a bijection $AC(I, M)\to (G \times L^2(I, \mathfrak{k}^{\perp}))/K$.
\end{proposition}
\begin{proof}
	Denote by $\pi: G\to M$ the quotient map, $V_g = \ker \pi_{*g}$ the vertical distribution and $H_g$ the horizontal distribution that is orthogonal to $V_g$ at $g\in G$. For every $g\in G$, $T_gG = V_g\oplus H_g$ and $\pi_{*g}$ restricts to an isomorphism between $H_g$ and $T_{\pi(g)}M$. Given $\beta\in AC(I, M)$ and $\alpha_0\in \pi^{-1}(\beta(0))$, there is a unique horizontal lift $\alpha$ starting at $\alpha_0$, that is,
	\begin{equation}
		\alpha\in AC^{\perp}(I, G),\quad\pi\circ\alpha = \beta,\quad \alpha(0)=\alpha_0.
	\end{equation}
	Now let $\alpha_0, \tilde{\alpha}_0$ be two lifts of $\beta(0)$, $\alpha$ and $\tilde{\alpha}$ be the unique horizontal lifts of $\beta$ starting at $\alpha_0$ and $\tilde{\alpha_0}$, respectively. Then $\alpha_0^{-1}\tilde{\alpha}_0\in K$. Let $y=\alpha_0^{-1}\tilde{\alpha}_0$. By right invariance of the metric on $G$, $\alpha y\in AC^{\perp}(I, G)$ and is also a lift of $\beta$ starting at $\tilde{\alpha}_0$. Thus $\tilde{\alpha} = \alpha y$. Let $(\alpha_0, q) = Q(\alpha)$ and $(\tilde{\alpha}_0, \tilde{q}) = Q(\tilde{\alpha})$. By computation, we have
	\begin{equation}
		(\tilde{\alpha}_0, \tilde{q}) = (\alpha_0, q)*y.
	\end{equation}
	From here, the statement follows.
\end{proof}
This identification is important because $G\times L^2(I,\mathfrak{k}^{\perp})$ has a natural product Riemannian structure. Furthermore, $K$ is compact and acts freely on this product by isometries, so the quotient has an inherited Riemannian metric. This inherited Riemannian metric is invariant under the left action of $G$ and the right action of $\tilde\Gamma$. By Proposition \ref{pro.Bi.M},  we can transfer the Riemannian metric on $(G\times L^2(I, \mathfrak{k}^{\perp}))/K$ to $AC(I, M)$, making the latter into a Riemannian manifold. Furthermore, the induced Riemannian metric on $AC(I, M)$ is invariant under the right action of $\tilde\Gamma$ and the left action of $G$.

Given $\beta_1, \beta_2\in AC(I, M)$, let $\alpha_1$ and $\alpha_2$ be horizontal lifts of $\beta_1$ and  $\beta_2$, respectively. Let
\begin{equation}
	Q(\alpha_1) = (\alpha_1(0), q_1),\quad Q(\alpha_2) = (\alpha_2(0), q_2).
\end{equation}
A minimal geodesic in the quotient $(G\times L^2(I, \mathfrak{k}^{\perp}))/K$ corresponds to a shortest geodesic between two orbits in $G\times L^2(I, \mathfrak{k}^{\perp})$ under the action of $K$. Thus the distance function on $AC(I, M)$ takes the form:
\begin{equation}\label{eq.disACM}
d(\beta_1, \beta_2) = \inf_{y\in K}\sqrt{d_G^2(\alpha_1(0), \alpha_2(0)y)+\|q_1-y^{-1}q_2y\|_{L^2}^2}.
\end{equation}
Consider now the right action of $\tilde{\Gamma}$ and the left action of $G$ on $G\times L^2(I, \mathfrak{k}^{\perp})$. We have the following proposition.
\begin{proposition}\label{pro.InvarianceDisACM}
	The distance function \eqref{eq.disACM} on $AC(I, M)$ is invariant under the action of $G$ and under the action of $\tilde{\Gamma}$.
\end{proposition}
\begin{proof}
	Since the distance function \eqref{eq.disACM} on $AC(I, M)$ is induced from the distance function \eqref{eq.disACG} on $AC^{\perp}(I, G)$, the proof is obvious.
\end{proof}


\subsection{Unparametrized curves with values in a homogeneous space}

Similarly to the case of unparametrized curves with values in Lie groups, we define an equivalence relation on $AC(I, M)$ as follows:
given $\beta_1$ and $\beta_2$, we say $\beta_1\sim\beta_2$ if and only if there is there exist $\bar{\beta}\in AC(I, M)$, $\gamma_1, \gamma_2 \in\tilde{\Gamma}$ such that 
$\beta_1=\bar{\beta}\circ\gamma_1$ and $\beta_2=\bar{\beta}\circ\gamma_2$ or $\overline{\beta_1\Gamma} = \overline{\beta_2\Gamma}$. Then we define the shape space $\mathcal{S}(I, M)$ as the set of equivalence classes under $\sim$:
\begin{equation}
\mathcal{S}(I,M) = AC(I, M)/\sim.
\end{equation}
The induced quotient distance on $\mathcal{S}(I, M)$ is as follows:
\begin{align}\label{eq.disSM}
d([\beta_1], [\beta_2]) &= \inf_{\gamma_1,\gamma_2\in\tilde{\Gamma}}d(\beta\circ\gamma_1, \beta_2\circ\gamma_2)\notag\\
& = \inf_{\substack{y\in K\\ \gamma_1,\gamma_2\in\tilde{\Gamma}}}\sqrt{d_G^2(\alpha_1(0),\alpha_2(0)y)+ \|q_1\star\gamma_1-y^{-1}(q_2\star\gamma_2)y\|_{L^2}^2},
\end{align}
where $(\alpha_1(0), q_1) = Q(\alpha_1),\, (\alpha_2(0), q_2) = Q(\alpha_2)$, $\alpha_1, \alpha_2\in AC(I,\mathfrak{k}^{\perp})$ are two horizontal lifts of $\beta_1$ and $\beta_2$, respectively.

We have the following theorem to show the non-existence of optimal reparametrizations between any two absolutely continuous curves with values in homogeneous spaces. 
\begin{thm}
	If $\dim M\geq 2$, there exists a pair of Lipschitz curves with values in $M$ such that the infimum \eqref{eq.disSM} cannot be obtained by any pair of reparametrizations.
\end{thm}
\begin{proof}
	To get such Lipschitz curves with values in $M$, we can project a pair of horizontal Lipschitz curves in $G$ to $M$. The construction of such Lipschitz curves in $G$ is the same as the construction stated in Theorem \ref{CounterExampleLieGroup} except that we set $\{w_1, w_2\}$ to be an orthogonal basis for a two dimensional subspace of $\mathfrak{k}^{\perp}$ instead of $\mathfrak{g}$.
\end{proof}

The results concerning the existence of optimal reparametrizations for curves with values in Lie groups can be extended to curves with values in homogeneous spaces. Before we state the main theorem, the definition of a symmetric space is needed.

\begin{defi}
	A Riemannian manifold $M$ is called a {\it symmetric space} if for every point $p\in M$ there exists an isometry $s_p: M\to M$ such that for any tangent vector $X\in T_pM$,
	\begin{equation}
	s_p(\operatorname{Exp}^M_p(tX)) = \operatorname{Exp}^M_p(-tX),
	\end{equation}
	where $\operatorname{Exp}^M_p: T_pM\to M$ is the Riemannian exponential map.
\end{defi}

In fact, a symmetric space is a special case of a Riemannian homogeneous space. The connected component of the isometry group of $M$ is a Lie group $G$ that acts on $M$ transitively. Denote by $K$ the subgroup of $G$ that fixes a point $p\in M$. Then there is a diffeomorphism $G/K\to M$. Under this diffeomorphism, the metric on $M$ corresponds to the metric on $G/K$ that is induced by a metric on $G$ that is left-invariant with respect to $G$ and bi-invariant with respect to $K$; see \cite{Mi2008}. 

Now, we give the main theorem of the existence of optimal reparametrizations for curves with values in homogeneous spaces.
\begin{thm}\label{OpExstHomogeneousMain}
	Let $M = G/K$ be a homogeneous space and let $\beta_1, \beta_2\in AC(I, M)$. Assume in addition that one of the following two conditions is satisfied:
	\begin{enumerate}[a)]
		\item $M$ is a symmetric space and one of $\beta_1$ or $\beta_2$ is a piecewise geodesic;
		\item $\beta_1, \beta_2\in C^1(I, M)$ with $\beta_1', \beta_2' \neq 0$ a.e.
	\end{enumerate}
	Then there exist $\gamma_1, \gamma_2\in\tilde{\Gamma}$ such that
	\begin{equation}
	d(\beta_1\circ\gamma_1, \beta_2\circ\gamma_2)=d([\beta_1], [\beta_2]).
	\end{equation}
	Furthermore, if $M$ is a symmetric space,  $\beta_1$ and $\beta_2$ are both piecewise geodesics, then this distance above can be realized by a pair of piecewise linear functions in $\tilde{\Gamma}$.
\end{thm}

Note that in the case of homogeneous spaces, the distance between $[\beta_1]$ and $[\beta_2]$ on $\mathcal{S}(I, M)$ is the infimum between the orbits of their horizontal lifts under the action of $K$ and $\tilde{\Gamma}$ both, which means we need to find not only the optimal reparametrizations but also the optimal $y_0\in K$. The following lemma provides the key ingredient to achieve this result.

\begin{lemma}\label{FyContinuous}
	Let $g_1, g_2\in G$ and $q_1, q_2\in L^2(I, \mathfrak{k}^{\perp})$. Then the map 
	$\mathcal{F}: K\to \mathbb{R}$ given by
	\begin{equation}
	\mathcal{F}(y)=\inf_{\gamma_1, \gamma_2\in\tilde{\Gamma}}\left(d_G^2(g_1, g_2y)+\|q_1\star\gamma_1-y^{-1}(q_2\star\gamma_2)y\|_{L^2}^2\right)
	\end{equation}
	is continuous. Furthermore, there exists $y_0\in K$ minimizing $\mathcal{F}$.
\end{lemma}
\begin{proof}
	Since $q\Gamma$ is dense in $q\tilde{\Gamma}$ and the distance function is reparametrization invariant, we have
	\begin{align}
	\mathcal{F}(y) &= \inf_{\gamma_1, \gamma_2\in\tilde{\Gamma}}\left(d_G^2(g_1, g_2y)+\|q_1\star\gamma_1-y^{-1}(q_2\star\gamma_2)y\|_{L^2}^2\right)\notag\\
	&= d_G^2(g_1, g_2y)+\inf_{\gamma_1, \gamma_2\in\tilde{\Gamma}}\|q_1\star\gamma_1-y^{-1}(q_2\star\gamma_2)y\|_{L^2}^2\notag\\
	&= d_G^2(g_1, g_2y)+\inf_{\gamma_1, \gamma_2\in\Gamma}\|q_1\star\gamma_1-y^{-1}(q_2\star\gamma_2)y\|_{L^2}^2\notag\\
	&= d_G^2(g_1, g_2y)+\inf_{\gamma\in\Gamma}\|q_1\star\gamma-y^{-1}q_2y\|_{L^2}^2.
	\end{align}
	It is easy to see that the first term $\mathcal{F}_1(y) = d_G^2(g_1, g_2y)$ is continuous. For $\mathcal{F}_2(y) = \inf_{\gamma\in\Gamma}\|q_1\star\gamma-y^{-1}q_2y\|_{L^2}^2$, we have 
	\begin{equation}
		\mathcal{F}_2 = \Phi_2\circ\Phi_1,
	\end{equation}
	where
	\begin{align}
		&\Phi_1: K \to \mathfrak{g}\notag\\
		&\Phi_1 (y) = y^{-1}q_2 y
	\end{align}
	and 
	\begin{align}
		&\Phi_2: L^2(I, \mathfrak{g})\to \mathbb R\notag\\
		&\Phi_2(u) = \inf_{\gamma\in\Gamma}\|q_1\star\gamma-u\|_{L^2}^2.
	\end{align}
	Since the multiplication and inversion maps on the Lie group $G$ are smooth, the map $G\times G\to G$ defined by $(y, g)\mapsto y^{-1}gy$ is smooth, thus the differential $K\times\mathfrak{g}\to\mathfrak{g}$, $(y, q)\mapsto y^{-1}qy$ is smooth. Therefore, $\Phi_1(y)$ is continuous. Also we have  
	\begin{align}
	\inf_{\gamma\in\Gamma}\|q_1\star\gamma-u_1\|_{L^2}
	\leq\|q_1\star\gamma-u_1\|_{L^2}\leq \|u_1-u_2\|_{L^2} + \|q_1\star\gamma-u_2\|_{L^2}
	\end{align}
	for any $\tilde{\gamma}\in \Gamma$. Taking the infimum over $\gamma\in\Gamma$, we obtain
	\begin{align}
	\Phi_2(u_1) - \Phi_2(u_2) \leq \|u_1-u_2\|_{L^2}.
	\end{align}
	Similarly,
	\begin{align}
	\Phi_2(u_2) - \Phi_2(u_1) \leq \|u_1-u_2\|_{L^2}.
	\end{align}
	Therefore, $\Phi_2(u)$ is continuous. Then the composition $\mathcal{F}_2(y)$ is continuous and thus $\mathcal{F}$ is continuous. Note that $K$ is compact. Thus there exists $y_0\in K$ realizing the minimum of $F(y)$.
\end{proof}

\begin{proposition}\label{OptimalExstHomogeneous}
	Let $g_1, g_2\in G$ and $q_1, q_2\in L^2(I, \mathfrak{k}^{\perp})$. Assume in addition that one of the following two conditions is satisfied:
	\begin{enumerate}[a)]
		\item $q_1$ or $q_2$ is a step map;
		\item $q_1$ and $q_2$ are continuous with $q_1, q_2\neq 0$ a.e.
	\end{enumerate}
	Then there exist $y_0\in K$, $\gamma_1, \gamma_2\in\tilde{\Gamma}$ such that
	\begin{multline}\label{eq.infdSquare}
	d_G^2(g_1, g_2y_0)+\|q_1\star\gamma_1-y_0^{-1}(q_2\star\gamma_2)y_0\|_{L^2}^2\\
	=\inf_{\substack{y\in K\\ \tilde{\gamma_1}, \tilde{\gamma_2}\in\tilde{\Gamma}}}d_G^2(g_1, g_2y)+\|q_1\star\tilde{\gamma_1}-y^{-1}(q_2\star\tilde{\gamma_2})y\|_{L^2}^2
	\end{multline}
	Furthermore, if $q_1$ and $q_2$ are both step maps, then this infimum can be realized by $y_0\in K$ and a pair of piecewise linear functions $\gamma_1, \gamma_2\in\tilde{\Gamma}$.
\end{proposition}
\begin{proof}
	Using Lemma \ref{FyContinuous}, there exists $y_0\in K$ minimizing the function
	\begin{align}
	F(y) = d_G^2(g_1, g_2y)+\inf_{\tilde{\gamma_1}, \tilde{\gamma_2}\in\tilde{\Gamma}}\|q_1\star\tilde{\gamma_1}-y^{-1}(q_2\star\tilde{\gamma_2})y\|_{L^2}^2.
	\end{align}
	For this particular optimal $y_0\in K$, by Proposition \ref{OptimalExstStepLieGroup}, we obtain $\gamma_1, \gamma_2\in\tilde{\Gamma}$ realizing the infimum of the second term of $F$ above and thus realizing the infimum in (\ref{eq.infdSquare}).
\end{proof}

For symmetric spaces, we have the following lemma.
\begin{lemma}\label{PGgPL}
	let $M = G/K$ be a symmetric space. Then every horizontal generalized PL curve with values in $G$ projects to a piecewise geodesic with values in $M$. Conversely, every piecewise geodesic with values in $M$ can be lifted to a horizontal generalized PL curve with values in $G$.
\end{lemma}
\begin{proof}
	Since the action of $G$ on $M$ by left multiplication is by isometries, every geodesic can be expressed in the form $g\operatorname{Exp}^M_K(tV)$, where $g\in G$, $V\in T_K M$. In addition, the submersion $\pi: G\to M, \pi(g) = gK$ is equivariant under the actions of $G$ on $G$ and $M$ by left multiplication, that is, for $g, x\in G$, $g\pi(x) = \pi(gx)$. By \cite[28.5.10]{Mi2008}, the result follows from $g\operatorname{Exp}_K^M (t \pi_{*} Y)= g \pi(\exp(t Y)) = \pi(g\exp(tY)),\ Y\in\mathfrak{k}^{\perp}$.
\end{proof}

We now give a proof of the main theorem in the case of homogeneous spaces.
\begin{proof}[Proof of Theorem \ref{OpExstHomogeneousMain}]
	By lifting $\beta_1, \beta_2$ to horizontal curves $\alpha_1, \alpha_2$ with values in the Lie group $G$, a) follows immediately from Proposition~\ref{OptimalExstHomogeneous}, Lemma \ref{PLLieGroupCorrespondence} and Lemma \ref{PGgPL}; for b), we obtain $\alpha_1, \alpha_2\in AC(I, \mathfrak{k}^{\perp})$ with $\alpha_1', \alpha_2' \neq 0$ a.e.. Then $q_1=q(\alpha_1)$ and $q_2=q(\alpha_2)$ are continuous with $q_1, q_2\neq 0$ a.e. in $\mathfrak{k}^{\perp}$. The results then follow from Proposition~\ref{OptimalExstHomogeneous}.
\end{proof}

\subsection{Other quotient spaces}

By Proposition \ref{pro.InvarianceDisACM}, $G$ acts on $AC(I, M)$ by isometries. Therefore, we can consider also the quotient space $AC(I, M)/G$. There the induced distance function is given by
\begin{equation}
d([\beta_1], [\beta_2]) = \inf_{y\in K,\ g\in G}\sqrt{d_G^2(\alpha_1(0), g\alpha_2(0)y)+\|q_1-y^{-1}q_2 y\|_{L^2}^2}.
\end{equation}
Note that for every $y\in K$, we can choose $g=\alpha_2(0)y\alpha_1^{-1}(0)$ such that $$d(\alpha_1(0), g\alpha_2(0)y)=0.$$ Thus we just need to find $y\in K$ minimizing $\|q_1-y^{-1}q_2 y\|_{L^2}$. This distance function can be simplified to 
\begin{equation}
d([\beta_1], [\beta_2])=\inf_{y\in K}\|q_1-y^{-1}q_2 y\|_{L^2}.
\end{equation} 

Now we consider both the actions of $\tilde{\Gamma}$ and $G$ on $AC(I, M)$. For the quotient space $\mathcal{S}(I, M)/G$, the induced distance function is given by
\begin{align}
d([\beta_1], [\beta_2])
=\inf_{\substack{y\in K\\ \gamma\in\Gamma, g\in G}}\sqrt{d_G^2(\alpha_1(0), g\alpha_2(0)y)+\|q_1-y^{-1}(q_2\star\gamma) y\|_{L^2}^2}.
\end{align}
Similarly, the distance function on $S(I, M)/G$ can be simplified to
\begin{align}
d([\beta_1], [\beta_2])
=\inf_{y\in K, \gamma\in\Gamma}\|q_1-y^{-1}(q_2\star\gamma) y\|_{L^2}.
\end{align}
\begin{remark}
A given \emph{Riemannian} manifold might be representable in multiple ways as a homogeneous space. Our construction of a Riemannian metric on the path space of curves with values in $M$ depends on the choice of representation. This is analogous to the framework of Celledoni et. al, as observed in \cite{CeEiEsSc2017a}.  
\end{remark}

\section{Implementation and Examples}
In this section we will describe the implementation of the proposed matching framework for  specific examples of homogeneous spaces, namely the $n$-dimensional sphere and the space of $n\times n$ positive definite symmetric matrices with determinant 1 ($\operatorname{PDSM}$). Note that the second example includes as an important special case the hyperbolic plane.  We will demonstrate the efficiency of the proposed numerical framework by computing minimizing geodesics in all of these cases and demonstrate the effects of the geometry of the manifold $M$ on the resulting optimal deformations. 

\subsection{The geodesic distance for parametrized curves}
We recall from Section 3 the formula for the geodesic distance on the space of parametrized curves with values in a homogeneous space $M$:
\begin{equation}\label{eq.disACM1}
	d(\beta_1, \beta_2) = \inf_{y\in K}\sqrt{d_G^2(\alpha_1(0), \alpha_2(0)y)+\|q_1-y^{-1}q_2y\|_{L^2}^2}.
\end{equation}
Here $\beta_1, \beta_2\in AC(I, M)$ are given curves, $\alpha_1$ and $\alpha_2$ are their horizontal lifts, and $(\alpha_1(0), q_1) = Q(\alpha_1),\, (\alpha_2(0), q_2) = Q(\alpha_2)$ are the corresponding square-root velocity functions. Before beginning a specific implementation, we develop some tools that will be helpful for computing the optimization over $K$ required by equation \ref{eq.disACM1} for {\it any} implementation of this method.

To compute the geodesic between $\beta_1$ and $\beta_2$, we need to compute the geodesic of minimal length between the orbits of $(\alpha_1(0), q_1)$ and $(\alpha_2(0), q_2)$ under the action of $K$. Since geodesics on $AC(I, G)$ can be calculated explicitly, we only need to find the element $y_0\in K$ that minimizes 
\begin{align}
d_G^2(\alpha_1(0), \alpha_2(0)y)+\|q_1-y^{-1}q_2y\|_{L^2}^2.
\end{align} 
Then the geodesic between
$(\alpha_1(0),q_1)$ and $(\alpha_2(0)y_0,y_0^{-1}q_2y_0)$ will project to a geodesic between $\beta_1$ and $\beta_2$ and thus we will obtain the induced geodesic distance, see \cite{Oneill1966,Mic2016} for more details regarding Riemannian submersions. 

To find the optimal $y\in K$ we employ a Riemannian gradient descent method. Therefore, we define $F: K\to \mathbb{R}$ by
\begin{equation}\label{eq:F}
F(y)=d_G^2(\alpha_1(0), \alpha_2(0)y)+\|q_1-y^{-1}q_2y\|_{L^2}^2.
\end{equation}
Since $K$ acts transitively on $(\alpha_2(0), q_2)\ast K$ we only need to calculate the gradient of $F$ at the identity.
We have:
\begin{lemma}
The gradient of the function $F: K\rightarrow \mathbb R$ at the identity is given by
\begin{align}
\nabla F=-2\operatorname{Proj}_{\mathfrak{k}}\left(\operatorname{Log}_I(\alpha_2(0)^{-1}\alpha_1(0))+\int_I \operatorname{ad}^T_{q_2}q_1dt\right),
\end{align}
where $\operatorname{Log}$ denotes the inverse Riemannian exponential at $\operatorname{Id}\in G$.
\end{lemma}

\begin{proof}
To calculate the gradient of $F$ we will consider the two terms of $F$ separately. We extend the first term of $F$ to a function $F_1: G\to\mathbb{R}$. Since the metric on $G$ is left invariant we have
\begin{equation}
	F_1(y) = d_G^2(\alpha_1(0), \alpha_2(0)y) = d_G^2(\alpha_2(0)^{-1}\alpha_1(0), y).
\end{equation}
By \cite[Theorem~7.1]{SrKl2016}, the gradient of $F_1$ at $y=\operatorname{Id}$ is given by 
\begin{equation}
	\nabla F_1 = -2\operatorname{Log}(\alpha_2(0)^{-1}\alpha_1(0)).
\end{equation}
If $\operatorname{Log}$ is multi-valued, we will take the value with the smallest norm. Restricting $F_1$ to $K$, we obtain the gradient of the first term of $F(y)$ at the identity to be
\begin{equation}
	-2\operatorname{Proj}_{\mathfrak{k}}\left(\operatorname{Log}(\alpha_2(0)^{-1}\alpha_1(0))\right).
\end{equation}
Denote by $F_2$ the second term of $F$. Using again the left invariance of the metric on $G$ we have
\begin{align}
F_2(y)&=\|q_1-y^{-1}q_2y\|_{L^2}^2=\|q_1\|_{L^2}^2+\|y^{-1}q_2y\|_{L^2}^2-2\langle q_1, y^{-1}q_2y\rangle_{L^2}\notag\\
&=\|q_1\|_{L^2}^2+\|q_2\|_{L^2}^2-2\langle q_1, y^{-1}q_2y\rangle_{L^2}.
\end{align}
For $v\in\mathfrak{k}$ the directional derivative of $F_2$ at $\operatorname{Id}$ in the direction $v$ is given by
\begin{align}
	\frac{d}{ds}|_{s=0}F_2(\exp(sv)) &= -2\frac{d}{ds}|_{s=0}\langle q_1, \exp(sv)^{-1}q_2\exp(sv)\rangle_{L^2}\notag\\
	& = -2\frac{d}{ds}|_{s=0}\int_I\langle q_1, \exp(sv)^{-1}q_2\exp(sv)\rangle dt\notag\\
	&=-2\int_I\langle q_1, \operatorname{ad}_{(-v)}q_2\rangle dt=-2\int_I\langle q_1, [-v , q_2]\rangle dt\notag\\
	&=-2\int_I\langle q_1, [q_2 , v]\rangle dt=-2\int_I\langle q_1, \operatorname{ad}_{q_2}v\rangle dt\notag\\
	&=-2\int_I\langle \operatorname{ad}^T_{q_2}q_1, v\rangle dt\notag\\
	&=\langle-2\int_I \operatorname{ad}^T_{q_2}q_1dt, v\rangle,
\end{align}
where $\operatorname{ad}^T_{X}:\mathfrak{g}\to\mathfrak{g}$ is the adjoint of $\operatorname{ad}_{X}$ with respect to the bilinear form $\langle\,,\,\rangle$ on $\mathfrak{g}$, i.e., $\langle\operatorname{ad}^T_XY, Z\rangle = \langle Y, \operatorname{ad}_XZ\rangle$. Therefore the gradient of the second term at the identity is given by
\begin{equation}
-2\operatorname{Proj}_{\mathfrak{k}}\left(\int_I \operatorname{ad}^T_{q_2}q_1dt\right).
\end{equation}
Hence we obtain the desired formula for the gradient.
\end{proof}
\begin{remark}
For $K$ a matrix Lie group with inner product given by $\langle u, v\rangle =\operatorname{tr}(uv^t)$, the gradient of $F$ at the identity simplifies to
\begin{align}
\nabla F=-2\operatorname{Proj}_{\mathfrak{k}}\left(\operatorname{Log}_I(\alpha_2(0)^{-1}\alpha_1(0))+\int_0^1(q_2^tq_1-q_1q_2^t)dt\right).
\end{align}
\end{remark}
Using the explicit formula of the gradient of $F$, it is straight-forward to implement a gradient descent based method to find a (local) minimizer of $F$.

\subsection{The geodesic distance for un-parametrized curves}
For unparametrized curves, the distance function on $\mathcal{S}(I, M)$ is given by
\begin{align*}
d([\beta_1], [\beta_2]) &= \inf_{\substack{y\in K\\ \gamma_1,\gamma_2\in\tilde{\Gamma}}}\sqrt{d_G^2(\alpha_1(0),\alpha_2(0)y)+ \|q_1\star\gamma_1-y^{-1}(q_2\star\gamma_2)y\|_{L^2}^2}
\end{align*}
where $(\alpha_1(0), q_1) = Q(\alpha_1),\, (\alpha_2(0), q_2) = Q(\alpha_2)$ and $\alpha_1, \alpha_2\in AC(I,\mathfrak{k}^{\perp})$ are two horizontal lifts of $\beta_1$ and $\beta_2$, respectively. Note that Theorem~\ref{OpExstHomogeneousMain} guarantees the existence of both optimal reparametrizations  $\gamma_1,\gamma_2\in \tilde \Gamma$ and $y_0\in K$. However, the solutions become highly non-unique, since for any diffeomorphisms $\gamma\in \Gamma$ and any pair of optimal reparametrizations 
$\gamma_1^*,\gamma_2^*\in \tilde \Gamma$ the pair $\gamma_1^*\circ\gamma,\gamma_2^*\circ\gamma\in \tilde \Gamma$ is also optimal. 

For the purpose of this article, we decided to solve the simpler problem
\begin{align}\label{unparamprob}
\inf_{y\in K, \gamma\in\Gamma}\sqrt{d_G^2(\alpha_1(0),\alpha_2(0)y)+ \|q_1-y^{-1}(q_2\star\gamma)y\|_{L^2}^2}
\end{align}
Using that $\Gamma$ is a dense subset of $\tilde \Gamma$, we can approximate the geodesic distance arbitrarily well using this approach. For each fixed $y_0\in K$ 
the problem reduces to the optimal reparametrization problem for the SRVF for curves with values in Euclidean spaces. This problem is well-studied and there exists a variety of different approaches to solve the optimization problem, c.f., \cite{SrKl2016,HuGaSrAb2016,BaEsGr2017}.  We choose to use dynamic programming, see \cite{Bell2010}, to approximate the reparametrization $\gamma$. Since the action of $K$ and the action of $\Gamma$ on $G\times L^2(I,\mathfrak{k}^{\perp})$ commute, we can iteratively use the gradient method and the dynamic programming algorithm to obtain a satisfactory approximation of the geodesic distance.
\begin{figure}[htbp]
	\begin{center}
		\includegraphics[width=0.23\linewidth]{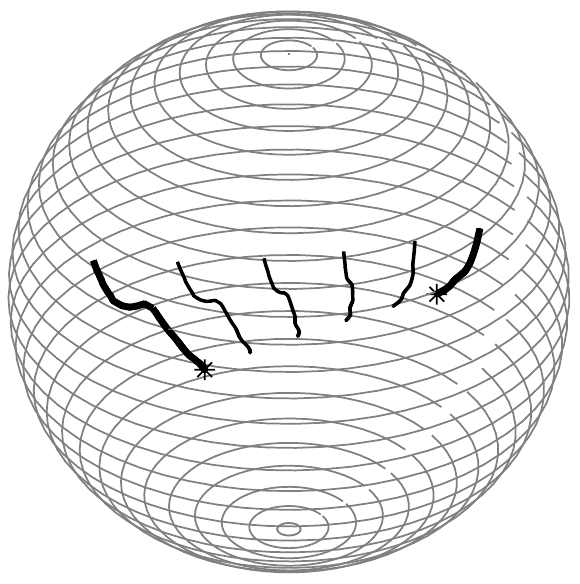}
		\includegraphics[width=0.23\linewidth]{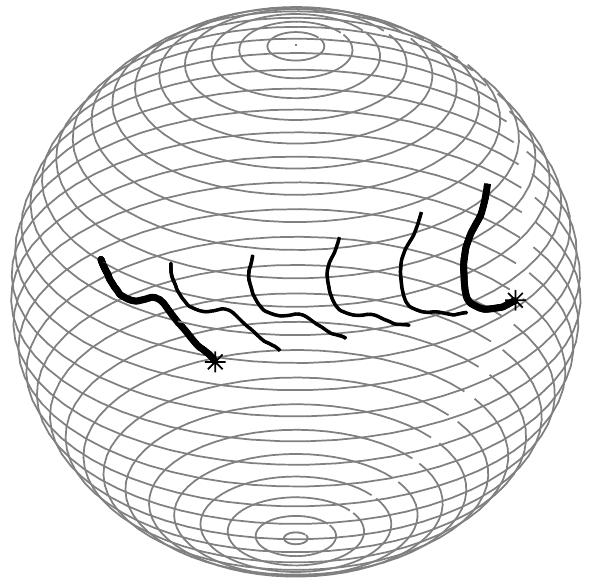}
		\includegraphics[width=0.23\linewidth]{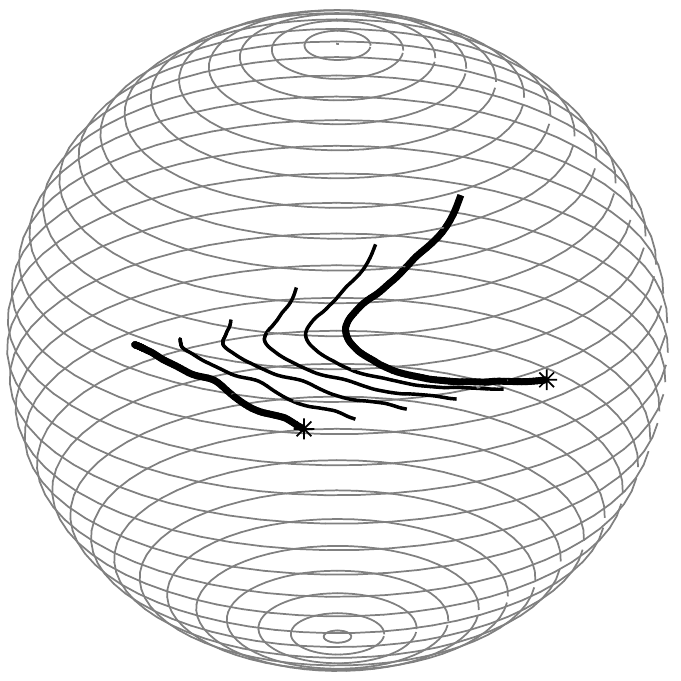}
		\includegraphics[width=0.23\linewidth]{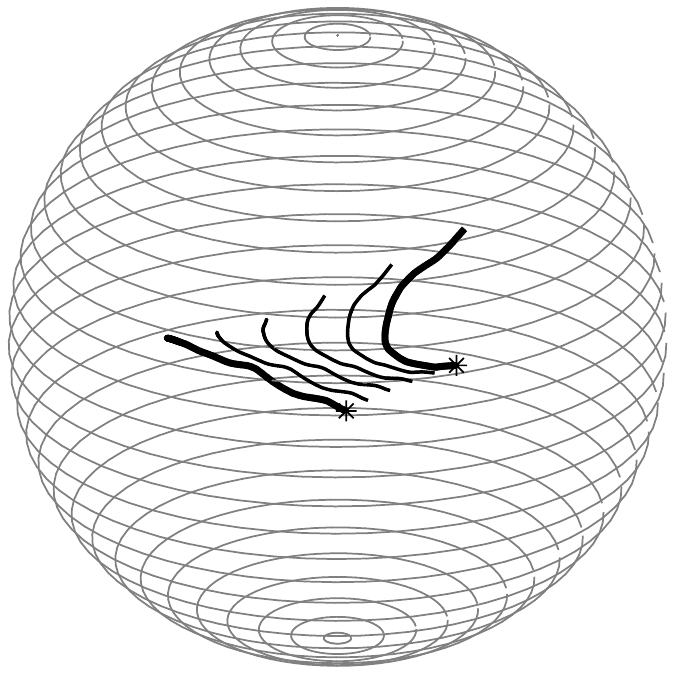}
	\end{center}
	\caption{Examples of minimizing geodesics between two curves in $\mathcal{S}(I, S^2)$. Starting points of the curves are marked with an~$\star$.}
	\label{fig.geodesic.S2}
\end{figure}

\subsection{Specific examples}
In this section we will present selected examples in the cases of  $M$ being the two-dimensional sphere, the hyperbolic plane, or the set of three by three positive definite symmetric matrices ($\operatorname{PDSM}$).

In Fig.~\ref{fig.geodesic.S2} we show four examples of minimizing geodesics on the sphere. The initial and target curves represent the shapes of hurricanes, taken from the National Hurricane Center website: \url{http://www.nhc.noaa.gov/data/}, where each hurricane track is discretized using 100 points. Statistical analysis of this data set using a previous adaptation of elastic shape analysis (TSRVF) can be found in the article \cite{SuKuKlSr2014}.  Analysis using the current method can be found in our earlier conference proceedings \cite{SuKlBa2017a}. In \ref{Appendix:A} of the current paper, we derive several of the specific formulas required for implementations involving a sphere $S^n$ of arbitrary dimension.

\begin{figure}[htbp]
	\begin{center}
		\includegraphics[width=0.233\linewidth]{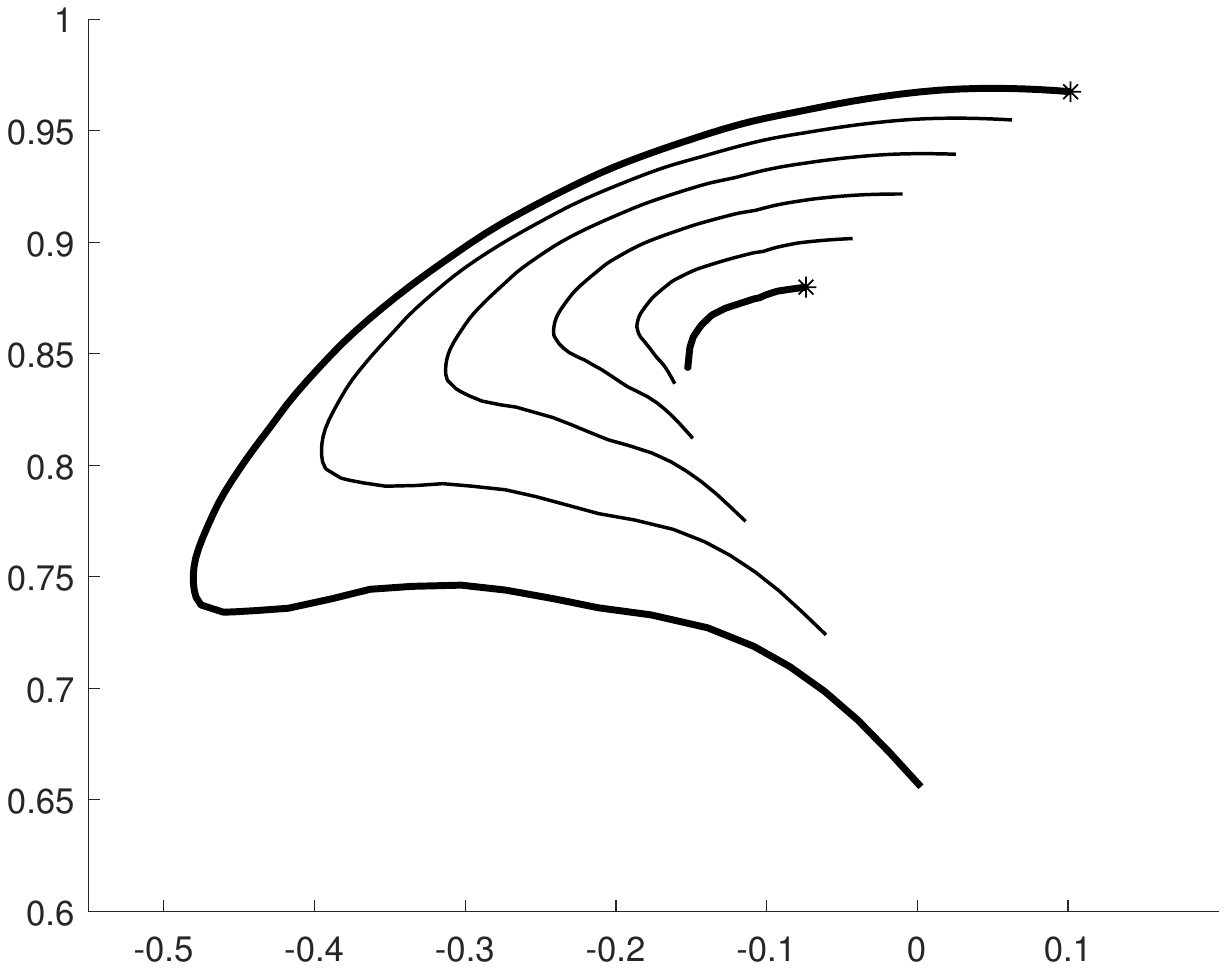}
		\includegraphics[width=0.233\linewidth]{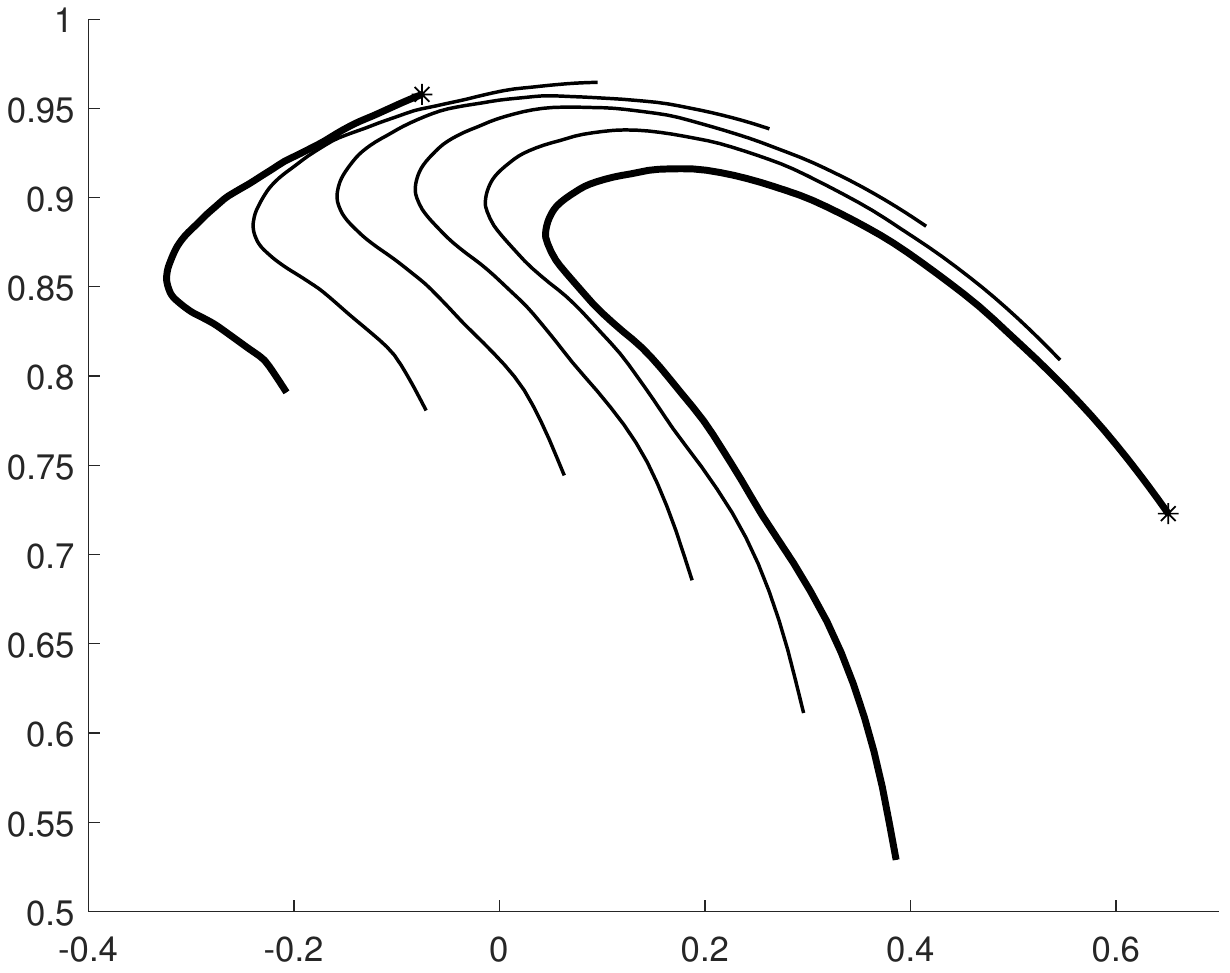}
		\includegraphics[width=0.233\linewidth]{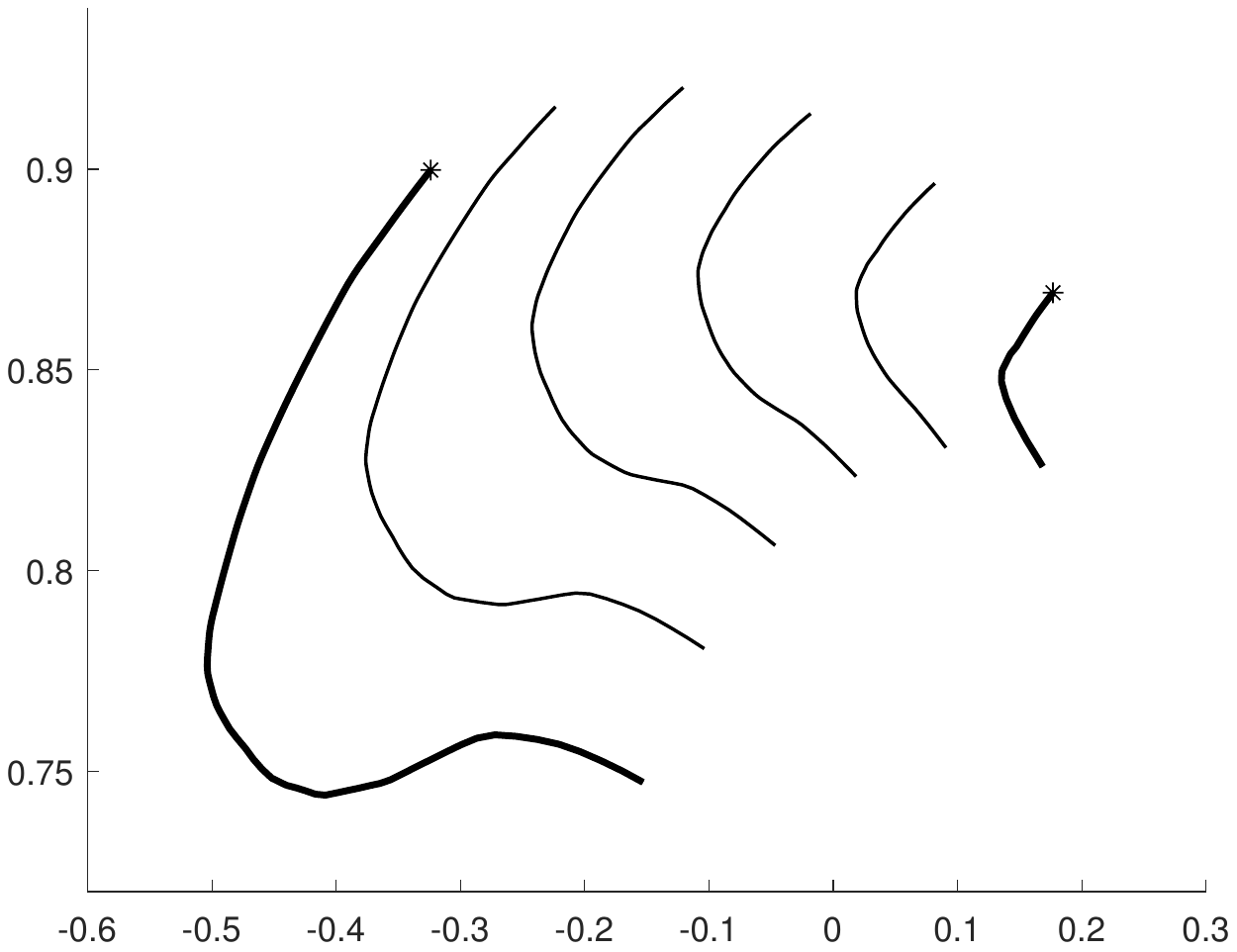}
		\includegraphics[width=0.233\linewidth]{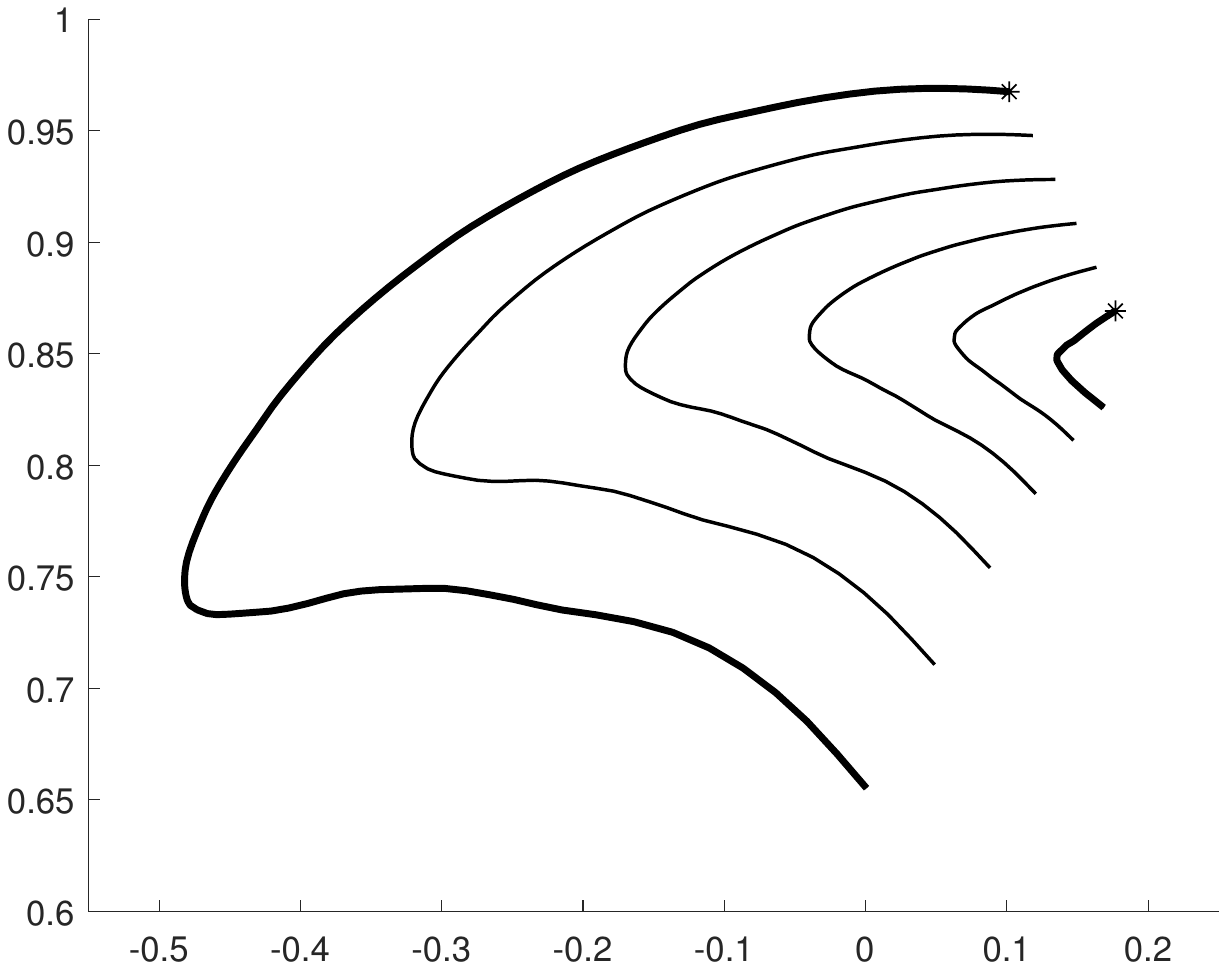}
	\end{center}
	\caption{Examples of minimizing geodesics between two curves in $\mathcal{S}(I, \mathbb H^2)$. Starting points of the curves are marked with an~$\star$.}
	\label{fig.geodesic.H2}
\end{figure}

As a second example we consider curves in the hyperbolic plane, see Fig.
\ref{fig.geodesic.H2}. Again, we discretize the curves using 100 points.  To visualize the curves, we used the upper-half plane model for hyperbolic space. This allows us to demonstrate the effect of the geometry of $M$ on the resulting optimal deformations and registrations. In Fig.~\ref{fig.H2&R2} we show the minimizing geodesic between curves, where we interpret these curves either as curves in the hyperbolic plane or as curves in Euclidean space. One can see that the choice of Riemannian metric has a large effect on the resulting geodesics. This serves as a strong motivation for the developed framework, as it suggests that one should not ignore the geometry of the ambient space for applications in shape analysis. In the example on the very left one can also see that a whole part of the first curve wants to be deformed to a single point on the second curve, which demonstrates the result that the optimal deformation is only an element of the closure $\tilde \Gamma$. The formulas for the hyperbolic plane and in particular the calculation of the inverse of the Riemannian exponential map is described in \ref{Appendix:B}.

\begin{figure}[htbp]
	\begin{center}
		\includegraphics[width=0.233\linewidth]{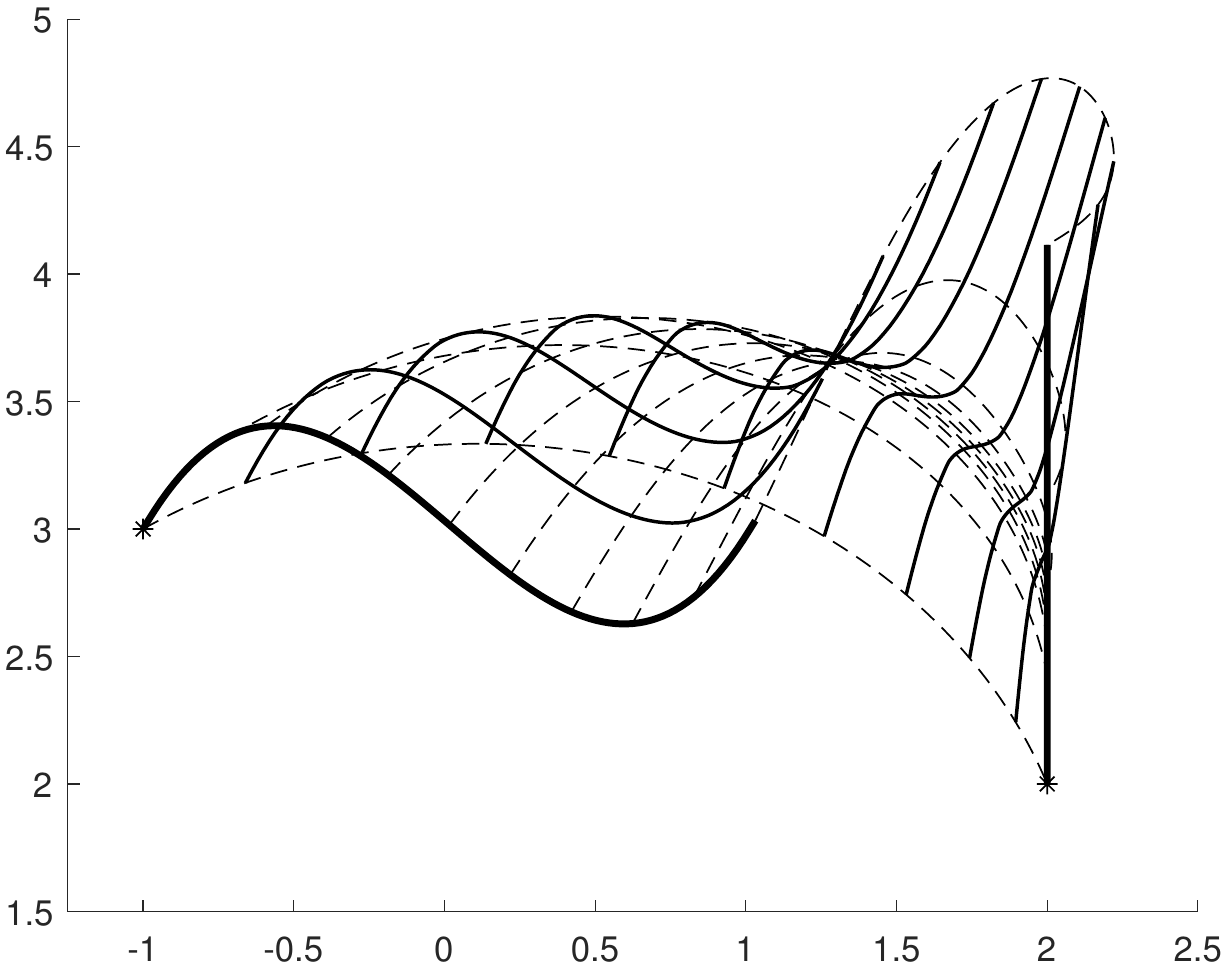}
		\includegraphics[width=0.233\linewidth]{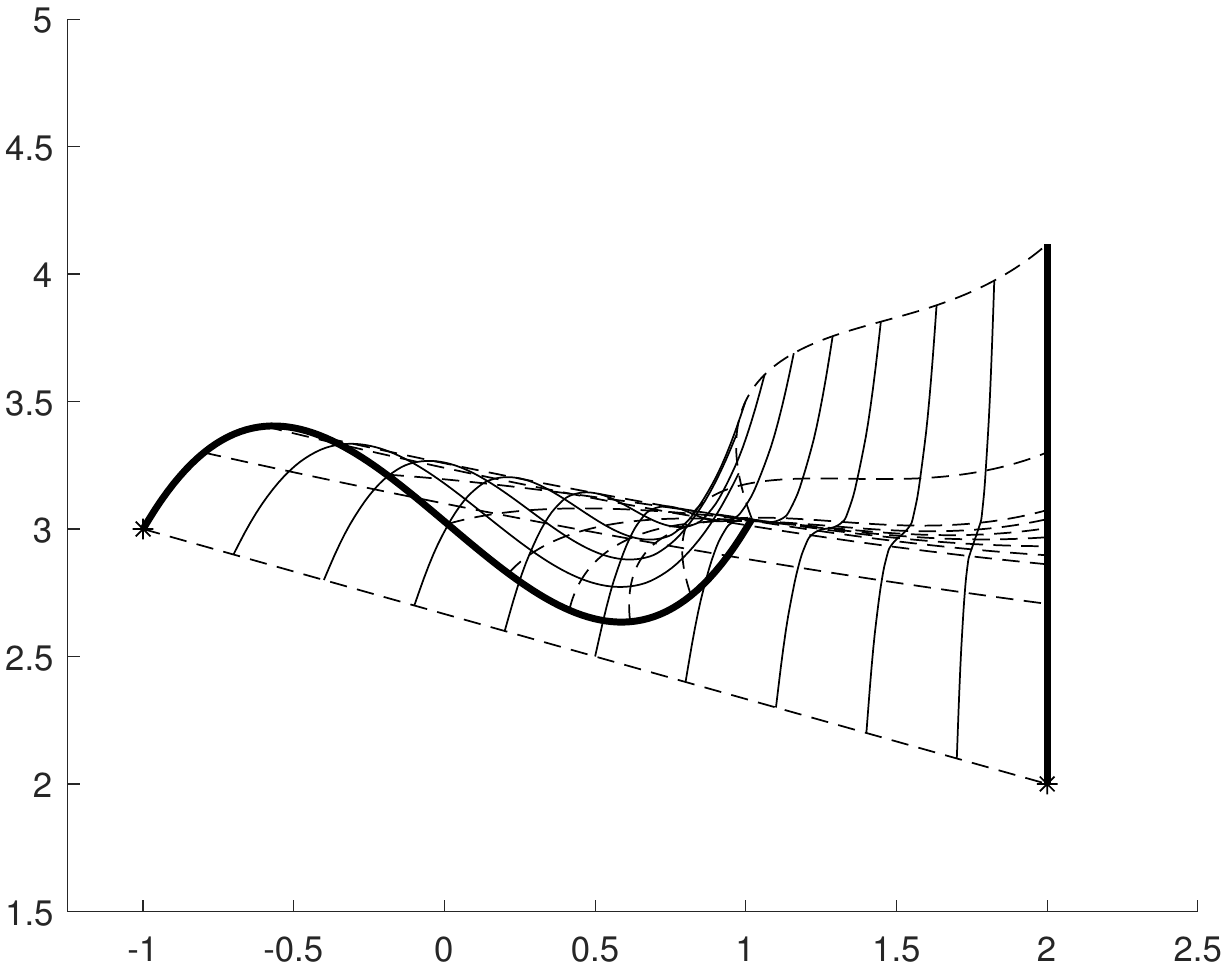}
		\includegraphics[width=0.233\linewidth]{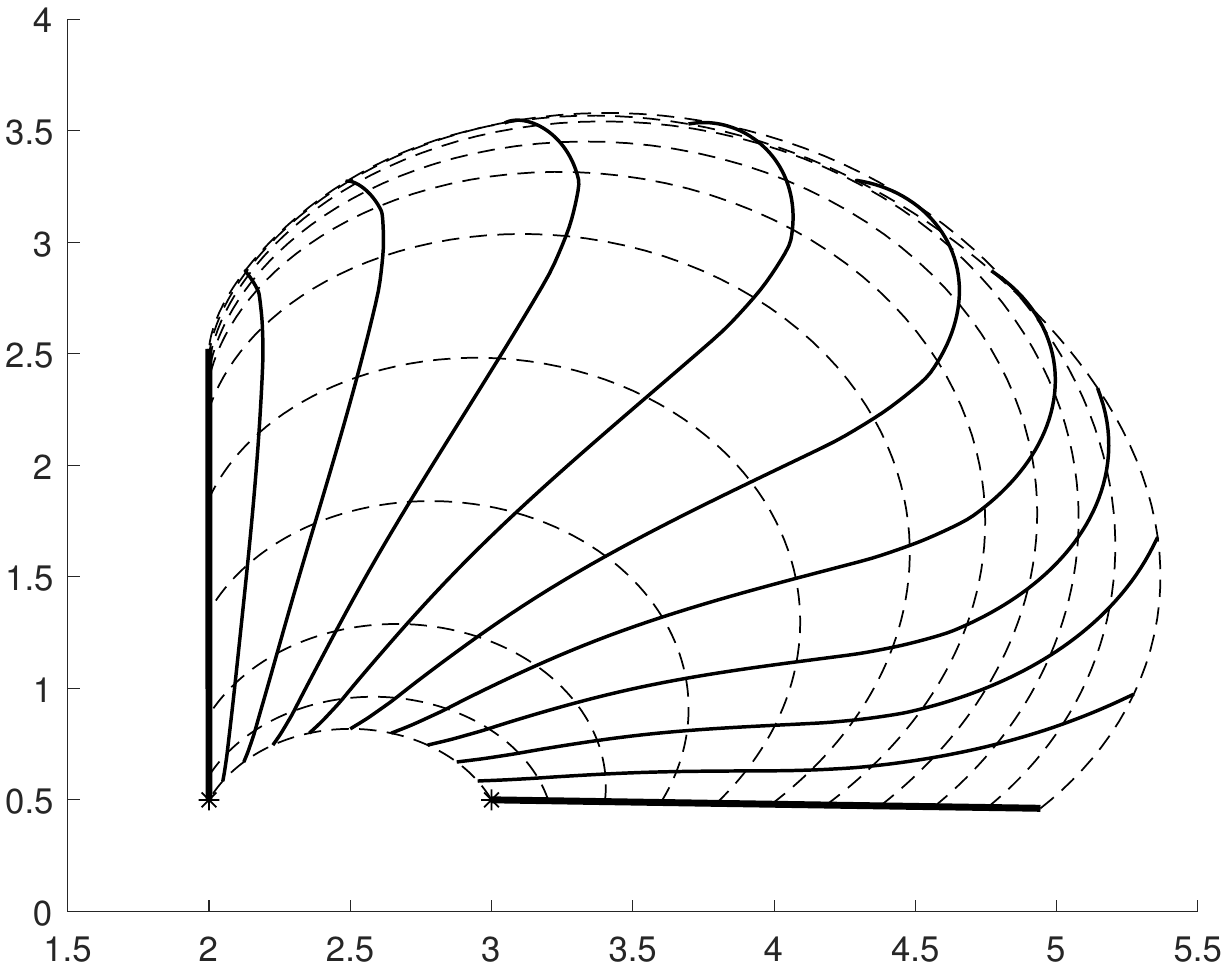}
		\includegraphics[width=0.233\linewidth]{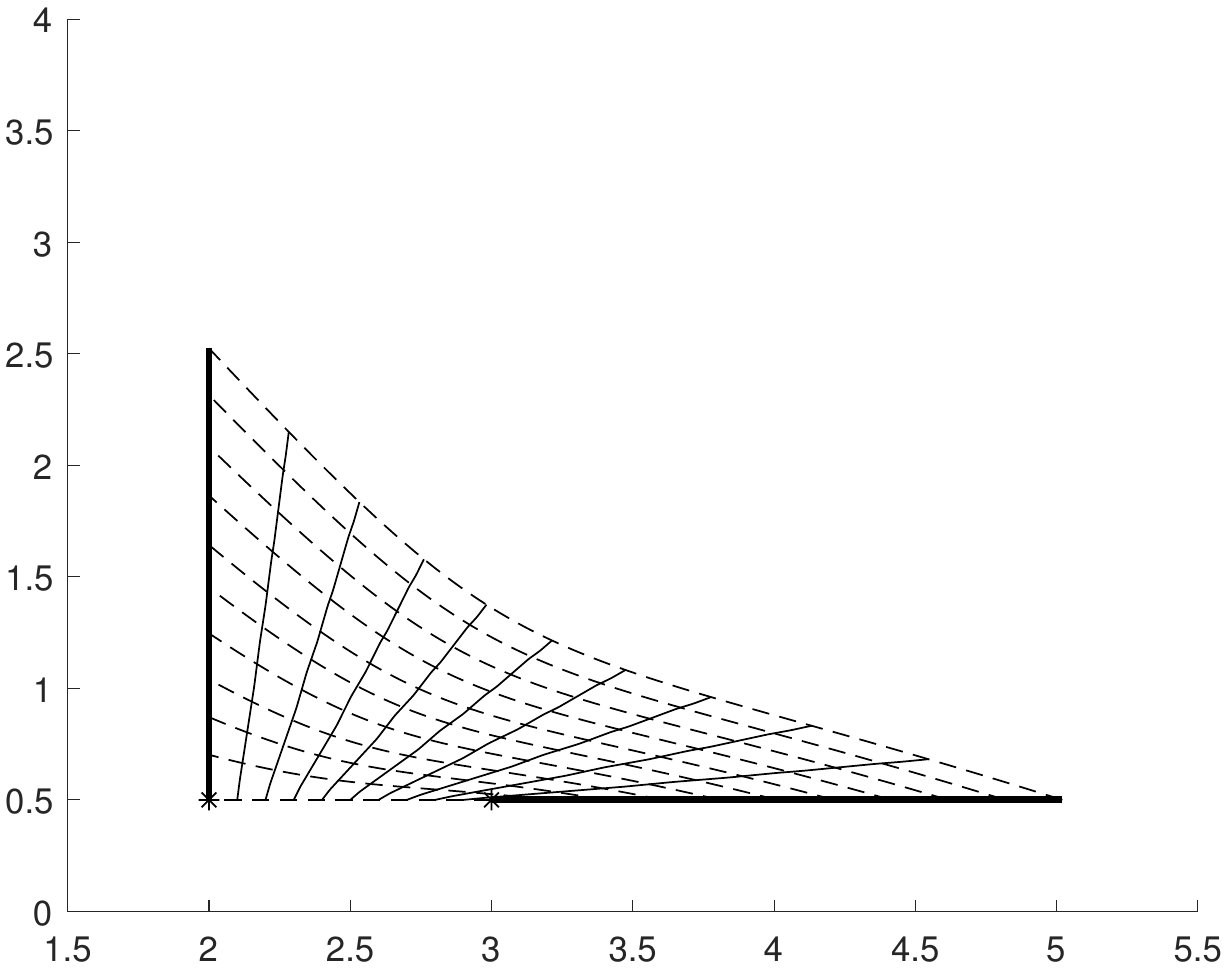}
	\end{center}
	\caption{Comparison between matching of curves in $\mathcal{S}(I, \mathbb H^2)$ and in $\mathcal{S}(I, \mathbb{R}^2)$. The solid lines show the inter-mediate curves, whereas the broken lines depict specific particle paths and can be used to visualized the optimal  registrations.}
	\label{fig.H2&R2}
\end{figure}

Finally in Fig.~\ref{fig.geodesic.PDSM} we show an example of minimizing geodesics between two curves in the space of $3\times 3$ positive definite symmetric matrices with determinant $1$ ($\operatorname{PDSM}_{3\times 3}$). Note that each $3\times 3$ positive definite symmetric matrix can be visualized as an ellipsoid with principal directions parallel to its eigenvectors and axes proportional to its eigenvalues. This enables us to use a sequence of ellipsoids to represent a discrete curve with values in $\operatorname{PDSM}_{3\times 3}$. The details of how our framework was applied on curves with values in $\operatorname{PDSM}_{3\times 3}$ can be found in \ref{Appendix:B}.

\begin{figure}
	\begin{center}
		\includegraphics[width=1\linewidth]{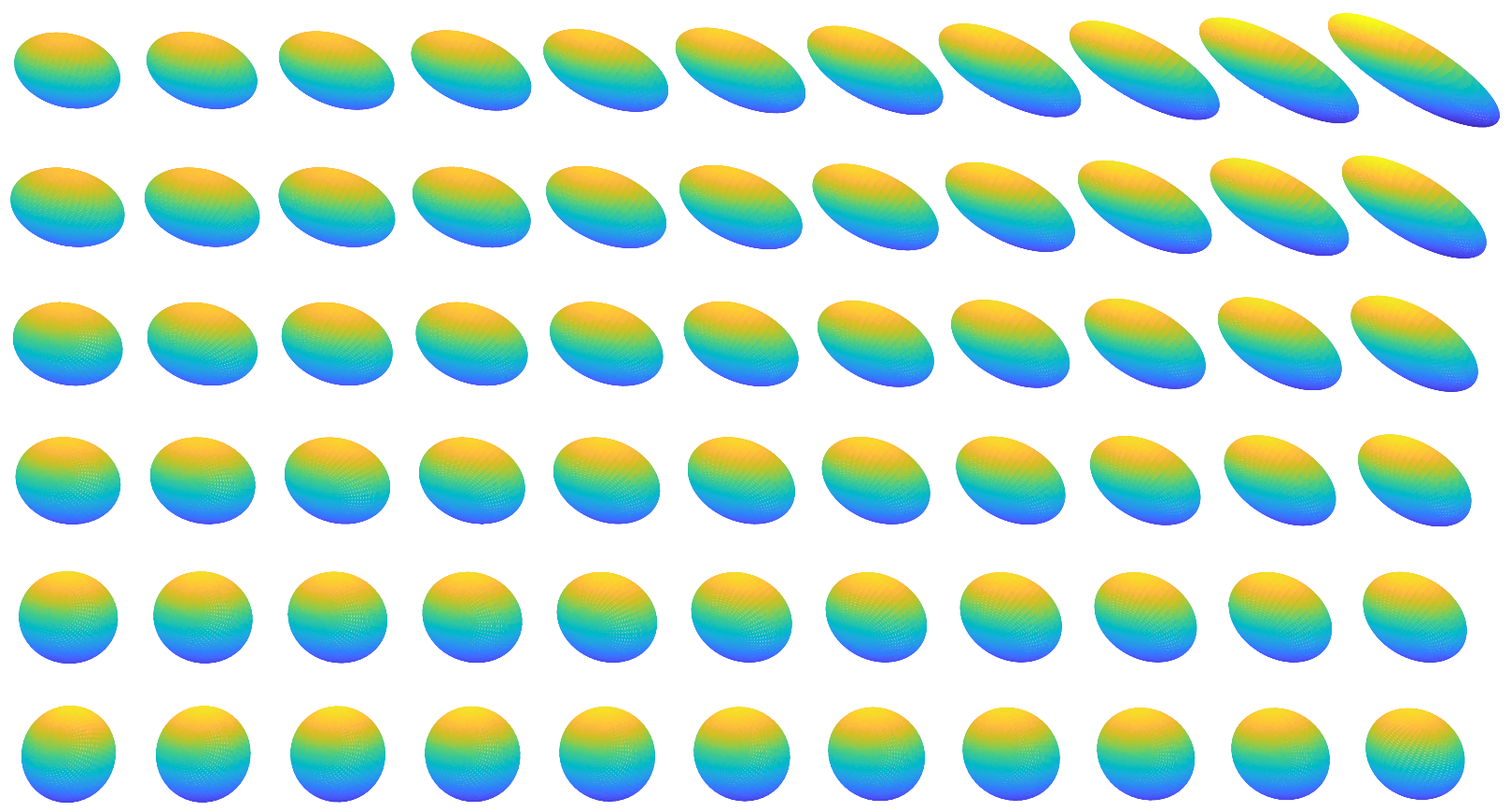}
	\end{center}
	\caption{An example of minimizing geodesic between two curves in $\mathcal{S}(I, \operatorname{PDSM}_{3\times 3})$. The right column and the left column of ellipsoids represent the boundary curves.} 
	\label{fig.geodesic.PDSM}
\end{figure}

To demonstrate the efficiency of our numerical framework we present the average time for matching a pair of curves in $\mathbb R^2$, $\mathbb H^2$, $S^2$ and in $\operatorname{PDSM}_{3\times 3}$. For the calculation of the average times we solved  for each case 1225 boundary value problems on an Intel Core i7-4510U (2.00GHz) machine. Note that passing from curves in Euclidean space to curves with values in a homogeneous space leads to a significantly slower performance. However, the obtained average time is still more satisfactorily fast; in particular, it is faster by an order of magnitude than the methods presented in \cite{ZhSuKlLeSr2015} and \cite{LeArBa2015}.

\begin{table}\label{time_results}
\begin{center}
	\begin{tabular}{ | c | c | c | c |}
		\hline
		{\multirow{2}{*}{Homogeneous Spaces} }  
		& \multicolumn{3}{c|}{\raisebox{.1cm}{\phantom{M}}\raisebox{-.1cm}{\phantom{M}}Average Time (seconds)} \\ 
		\cline{2-4}
		& \raisebox{.1cm}{\phantom{M}}\raisebox{-.1cm}{\phantom{M}}100 points & 300 points & 500 points \\ 
		\hline
		$\raisebox{.1cm}{\phantom{M}}\raisebox{-.1cm}{\phantom{M}}\mathbb R^2$ & 0.012314 & 0.096156 & 0.278207  \\ 
		\hline
		$\raisebox{.1cm}{\phantom{M}}\raisebox{-.1cm}{\phantom{M}}\mathbb H^2$ (evaluation) & 1.013337 & 1.103829 & 1.303323  \\ \hline
		$\raisebox{.1cm}{\phantom{M}}\raisebox{-.1cm}{\phantom{M}}\mathbb H^2$ (gradient method) & 0.626906 & 1.708868 & 2.715912  \\ 
		\hline
		$\raisebox{.1cm}{\phantom{M}}\raisebox{-.1cm}{\phantom{M}}S^2$ (evaluation) & 0.016318 & 0.110217 & 0.292127  \\ \hline
		$\raisebox{.1cm}{\phantom{M}}\raisebox{-.1cm}{\phantom{M}}S^2$ (gradient method) & 0.104503 & 0.376404 & 0.733898  \\ 
		\hline
		$\raisebox{.1cm}{\phantom{M}}\raisebox{-.1cm}{\phantom{M}} \operatorname{PDSM}_{3\times 3}$ & 1.002971 & 2.359541  & 3.914588 \\ \hline
	\end{tabular}
	\end{center}
	\caption{Average computation time for solving the geodesic boundary value problem on the space of unparametrized curves.}
\end{table}

\appendix

\section{The homogeneous space $S^n$}\label{Appendix:A}
To view the $n$-dimensional sphere as a homogeneous space, we represent it as the quotient space $S^n\cong\operatorname{SO}(n+1)/\operatorname{SO}(n)$, where $\operatorname{SO}(n)$ denotes the special orthogonal group
\begin{equation}
\operatorname{SO}(n)=\{A\in GL(n,\mathbb{R})\,|\, A^tA=AA^t=I, \det(A)=1\}
\end{equation} with corresponding Lie algebra
\begin{equation}
\mathfrak{so}(n)=\{X\in M(n,\mathbb{R})\,|\,X+X^t=0\}.
\end{equation}
Let ${\bf n}=(0,...0,1)^t\in S^n$ be the north pole of the sphere. We identify $\operatorname{SO}(n)$ as a subgroup of $\operatorname{SO}(n+1)$ using the inclusion
$$	A \to \left ( \begin{array}{lcr}
A & 0 \\ 0 & 1
\end{array}  \right ).$$
The quotient map $\pi: \operatorname{SO}(n+1)\to\operatorname{SO}(n+1)/\operatorname{SO}(n)\cong S^n$ is then given by $\pi(\alpha) = \alpha{\bf n}$.

We use the Riemannian metric 
\begin{equation}
	\langle u, v\rangle_g = \operatorname{tr}(uv^t).
\end{equation}
on $\operatorname{SO}(n+1)$.
It is straightforward to check that this metric is bi-invariant with respect to $\operatorname{SO}(n+1)$ and thus in particular with respect to $\operatorname{SO}(n)\subset \operatorname{SO}(n+1)$. Thus, the metric descends to a Riemannian metric on the quotient space and it turns out that this metric is equal to the standard metric on $S^n$. Furthermore, using the bi-invariance of the metric, the Riemannian exponential on $T_I\operatorname{SO}(n+1)$  is equal to the Lie group exponential \cite{Petersen1998} and is thus of the form $v\to \exp(v)$, where $\exp$ denotes the matrix exponential. The inverse Riemannian exponential map at the identity is the matrix log function $g\to \log(g)$. 

The following well-known lemma is needed in calculating the horizontal lifts of a curve with values on $S^n$:

\begin{lemma}\label{lem.Sn.effiRotation}
	If $p, q\in S^n$ and $p\neq-q$, then the most efficient rotation that takes $p$ to $q$ can be expressed as
	\begin{equation}\label{rot:formula}
	R_{p, q}=\left(I-\dfrac{2}{|p+q|^2}(p+q)(p^t+q^t)\right)(I-2pp^t).
	\end{equation}
	By most efficient, we mean the rotation closest to $I$ with respect to the bi-invariant metric on $\operatorname{SO}(n+1)$.
\end{lemma}
\begin{proof}
	Here we just give a proof of the simplest case, that is, $n=1$. The result can be generalized to $S^n$ for $n\geq 2$. Let $p\in S^1 \subset \mathbb{R}^2$. Using basic Euclidean geometry, the operator that reflects through the orthogonal complement of $p$ is of the form $\rho_p= I- 2pp^t$. Suppose $p, q\in S^1$, then
	the formula of the most efficient rotation is
	\begin{equation}
	R_{p,q} = \rho_{\frac{p+q}{\|p+q\|}}\circ\rho_p.
	\end{equation}
	By most efficient, we mean the rotation with the smallest angle, assumed by $\theta$. Note that the geodesic from the identity $I\in\operatorname{SO}(2)$ to $R_{p,q}$ is of the form
	$\exp\begin{pmatrix}
	0 & -\theta\\
	\theta & 0
	\end{pmatrix}$. The result follows from $d(I, R_{p,q}) = \left\|\begin{matrix}
	0 & -\theta\\
	\theta & 0
	\end{matrix}\right\| =\theta\sqrt{2}.$
\end{proof}

This formula is only valid if $p\neq -q$, since if $p = -q$ there is no unique shortest rotation taking $p$ to $q$.

Now we will use the above lemma to find a discrete horizontal lift $\alpha$ of a (discrete) curve $\beta: I \to S^n$. Suppose we are given the values of $\beta(t)$ sampled at $N+1$ equidistant points of $I$, i.e., we are given $\{\beta(t_i)\}$ for $t_i = \frac{i}{N},\, i = 0, 1,\cdots,N$ and we assume that $\beta(t)$ be a piecewise geodesic connecting $\beta(t_i)$. We want to find points $\alpha(t_i)\in\operatorname{SO}(n+1)$ such that
\begin{enumerate}[(a)]
	\item\label{Sn.propertyLift(a)} $\beta(t)=\pi(\alpha(t))$ for all $t\in I$, where $\alpha(t)$ is the generalized PL curve connecting the points $\alpha(t_i)$;
	\item\label{Sn.propertyLift(b)} $\alpha'(t)\perp \alpha(t)\mathfrak{so}(n)$ for all $t\in I$.
\end{enumerate}
We have the following algorithm to calculate $\alpha\in AC^{\perp}(I, \operatorname{SO}(n+1))$:
\begin{enumerate}[(1)]
	\item For
	$\beta(0)\neq-{\bf n}$, let $\alpha(0)=R_{{\bf n},\beta(0)}$. If $\beta(0)=-{\bf n}$, let 
	$$
	\alpha(0)=\left(  \begin{array}{lcr}
	-1 & 0 & 0\\0 & I_{n-1} & 0\\0 & 0 & -1
	\end{array} \right),
	$$ where $I_{n-1}$ is the $(n-1)\times(n-1)$ identity matrix. 
	\item Given $\alpha(t_i)$, set $\alpha(t_{i+1})=R_{\beta(t_i), \beta(t_{i+1})}\alpha(t_i)$.
\end{enumerate}
It is easy to see that $\alpha(t_i)$ satisfies the first condition (\ref{Sn.propertyLift(a)}) above, i.e.,
$\pi(\alpha(t_i))=\beta(t_i)$, for all $i\in \{1,...,N\}$.
It remains to check that the discrete form of the second condition (\ref{Sn.propertyLift(b)})  holds, that is, the geodesic between $\alpha(t_i)$ and $\alpha(t_{i+1})$ is perpendicular to the orbits with respect to these two elements. Assume that $\alpha(t_{i+1})=B\alpha(t_i)$ for $B\in \operatorname{SO}(n+1)$. By the bi-invariance of the metric, we have the distance $d(\alpha(t_i), \alpha(t_{i+1}))=d(\alpha(t_i), B\alpha(t_i))=d(I, B).$ It is easy to see that $B$ left translates the orbit $\alpha(t_i)$ to the orbit $\alpha(t_{i+1})$, which is equivalent to left translating $\beta(t_i)$ to $\beta(t_{i+1})$, that is, $B\beta(t_i)=\beta(t_{i+1})$. By Lemma \ref{lem.Sn.effiRotation}, we know that $R_{\beta(t_i),\beta(t_{i+1})}\in \operatorname{SO}(n+1)$ is the most efficient rotation such that 
$d(\alpha(t_i), R_{\beta(t_i),\beta(t_{i+1})}\alpha(t_i))=d(I, R_{\beta(t_i),\beta(t_{i+1})})$
is smallest, which means the distance between $\alpha(t_i)$ and $R_{\beta(t_i),\beta(t_{i+1})}\alpha(t_i)$ realizes the shortest possible distance between all pairs of representatives of these two orbits. 

\begin{remark}
We have now described our method to calculate the horizontal lifts of curves in $S^n$. To calculate the geodesic it remains to solve the 
optimization problem \eqref{eq.disSM}. 
In the case of $S^2 =\operatorname{SO}(3)/\operatorname{SO}(2)$ this is a minimization over the one-dimensional compact group $\operatorname{SO}(2)$. Thus we can use, as an alternative to the gradient method, an evaluation based method to find the optimal $y\in\operatorname{SO}(2)$, i.e., 
discretize the one-dimensional compact group by a finite number of points and find the optimal element of this discretization.
\end{remark}


\section{The homogeneous space of all positive definite symmetric matrices}\label{Appendix:B}
We now describe the space of $n\times n$ positive definite symmetric matrices  with determinant one equipped with a natural metric. The main complication, as compared to the case $M=S^n$, will be the lack of an explicit formula for the inverse exponential map. 

To view the set of all $n\times n$ positive definite symmetric matrices with determinant  one as a homogeneous space we start by considering the special linear group
\begin{equation}
	\operatorname{SL}(n, \mathbb{R}) = \{A\in GL(n,\mathbb R)\,|\, \det(A) = 1\}
\end{equation}
with corresponding Lie algebra
\begin{equation}
	\mathfrak{sl}(n, \mathbb R) = \{X\in M(n, \mathbb)\,|\,\operatorname{tr}(X) = 0\}
\end{equation}
Let $\operatorname{PDSM}_{n\times n}$ denote the space of real $n\times n$ positive definite symmetric matrices with unit determinant. By the polar decomposition theorem, each matrix $B\in \operatorname{SL}(n,\mathbb{R})$ has a unique expression as $B=PV$, where $P\in \operatorname{PDSM}_{n\times n}$ and $V\in \operatorname{SO}(n)$. We  define the map $$\psi: \operatorname{SL}(n,\mathbb{R})\to \operatorname{PDSM}_{n\times n}$$ by $\psi(B)=P$, where $B = PV$ is the polar decomposition of $B$. An explicit formula for $\psi$ is $\psi(B)=\sqrt{BB'}$, where $\sqrt{\phantom{B} }$ denotes the unique positive definite symmetric square root. Clearly $\psi$ induces a diffeomorphism $\operatorname{SL}(n,\mathbb{R})/\operatorname{SO}(n)\to \operatorname{PDSM}_{n\times n}$.
We will now define a Riemannian metric on these spaces: for $u, v\in T_g\operatorname{SL}(n, \mathbb R)$ with $g\in\operatorname{SL}(n, \mathbb R)$ let
\begin{equation}
	\langle u, v\rangle_g = \operatorname{tr}\left(g^{-1}u\; (g^{-1}v)^t\right),
\end{equation}
where $A^t$ denotes the transpose of a matrix $A$.
It is easy to verify that this is a left invariant metric on  $\operatorname{SL}(n, \mathbb{R})$, which is right invariant with respect to $\operatorname{SO}(n)$. The Riemannian exponential map at the identity is of the form $v\to \exp(v^t)\exp(v-v^t)$, where $\exp$ denotes the standard matrix exponential, see \cite[Theorem~2.14]{AnLaReVa2014}. However, to our knowledge there exists no explicit formula for the inverse Riemannian exponential map. 

In the following we will describe an algorithm to numerically approximate it:
denote by $\operatorname{Exp}: \mathfrak{sl}(n, \mathbb{R})\to\operatorname{SL}(n, \mathbb R)$ the Riemannian exponential map and by $\operatorname{Log}: \operatorname{SL}(n, \mathbb R)\to \mathfrak{sl}(n, \mathbb{R})$ the inverse Riemannian exponential map (both at the identity).  For every $B\in\operatorname{SL}(n, \mathbb{R})$ we  define the function $F: B\cdot \operatorname{SO}(n)\to B \cdot\operatorname{SO}(n)$ by
\begin{align}
	F(z)=z\exp(\log(z^t)-\log(z)).
\end{align}
Here $\exp$ (resp. $\log$) denotes the matrix exponential (resp. matrix logarithm), which can be computed easily.
The important property of the map $F$ is as follows:  if $F(z_0) = B$ for some $z_0\in B\cdot \operatorname{SO}(n)$, then $X = \log(z_0^t)$ is the inverse Riemannian exponential of $B$, that is, $X = \operatorname{Log}(B)$. 

To solve this problem numerically we need to calculate the differential of the function $F$ at $z$:
\begin{equation}
	F_{*z}: T_z(B\cdot\operatorname{SO}(n))\to T_{F(z)}(B\cdot\operatorname{SO}(n)).
\end{equation}
The tangent space of the orbit $B\cdot\operatorname{SO}(n)$ for $B\in\operatorname{SL}(n, \mathbb R)$ is isomorphic to the Lie algebra $\mathfrak{so}(n)$ by left translation. We use this identification to compute
\begin{equation}
	D_zF: \mathfrak{so}(n)\to T_z(B\cdot\operatorname{SO}(n))\to T_{F(z)}(B\cdot\operatorname{SO}(n))\to\mathfrak{so}(n),
\end{equation}
which  can be approximated using finite differences:
\begin{equation}
	D_zF(x_i)=\operatorname{Proj}_{\mathfrak{so}(n)}\left(F(z)^{-1}\dfrac{F(z\exp(\delta x_i))-F(z)}{\delta}\right),
\end{equation}
where $\delta>0$ and $\{x_i\}$ denotes a basis of $\mathfrak{so}(n)$. Thus we obtain the following algorithm for the computation of the inverse Riemannian exponential map:
\begin{enumerate}[(1)]
\item Given $B_0\in\operatorname{SL}(n, \mathbb{R})$, set $z=\sqrt{B_0B_0^t}.$
	\item Compute $v=\log({F(z)^{-1}B_0})$.
	\item If $\|v\|$ is small, then stop. Otherwise update $z$ by $z=z\exp{\epsilon x}$, where $x=D_zF^{-1}(v)$ and $\epsilon>0$ is a chosen step size, and go back to step 1.
	\item Let $X_0 = \operatorname{Log}(B_0)=\log(z^t)$.
\end{enumerate}
Note that we are assuming that $D_zF$ is invertible, which turned out to be true in all of our numerical examples, but which we do not know how to prove.
To calculate the horizontal lifts of a curve with values in $\operatorname{PDSM}_{n\times n}$ we will make use of the following lemma:
\begin{lemma}\label{lem.PDSM.effi}
	Given $B_1, B_2\in \operatorname{SL}(n, \mathbb{R})$, the element in the orbit $B_2\cdot \operatorname{SO}(n)$ that is closest to $B_1$ is of the form $B_1\sqrt{B_1^{-1}B_2(B_1^{-1}B_2)^t}$.
\end{lemma}
\begin{proof}
We first consider the case $B_1 = I$. A geodesic in $\operatorname{SL}(n,\mathbb{R})$ from $I$ to an element in $B_2\cdot \operatorname{SO}(n)$ that is perpendicular to the orbit $B_2\cdot \operatorname{SO}(n)$ will be perpendicular to all $\operatorname{SO}(n)$ orbits it encounters. In particular, it will be perpendicular to    $\operatorname{SO}(n)$ at $I$; thus, the corresponding tangent vector at $I$ is a symmetric matrix. Since the geodesic in $\operatorname{SL}(n,\mathbb{R})$ starting from $I$ with direction $u$ is of the form $\exp(tu^t)\exp(t(u-u^t))$, it follows that the shortest geodesic from $I$ to $B_2\cdot \operatorname{SO}(n)$ is of the form $t\to\exp(tu)$, where $u$ is a symmetric matrix in $\mathfrak{sl}(n, \mathbb{R})$.	Thus the whole geodesic consists of symmetric matrices.  As a consequence the matrix in $B_2\cdot \operatorname{SO}(n)$ that is closest to $I$ will be the unique symmetric matrix in the orbit $B_2\cdot \operatorname{SO}(n)$, that is precisely the symmetric matrix $P=\sqrt{B_2B_2^t}$ that appears in the polar decomposition of $B_2=PV$, where $P\in \operatorname{PDSM}_{n\times n}$ and $V\in \operatorname{SO}(n)$. Now let $B_1$ be any arbitrary element in $\operatorname{SL}(n, \mathbb{R})$. Since the metric on $\operatorname{SL}(n, \mathbb{R})$ is left invariant, we first find the element in the orbit $B_1^{-1}B_2\cdot\operatorname{SO}(n)$ that is closest to $I$, and then left translate it by $B_1$. Hence, the element in the orbit $B_2\cdot\operatorname{SO}(n)$ that is closest to $B_1$ is simply $B_1\sqrt{B_1^{-1}B_2(B_1^{-1}B_2)^t}$.
\end{proof}

We will now use the above lemma to find a discrete horizontal lift $\alpha$ of a (discrete) curve $\beta: I\to \operatorname{PDSM}_{n\times n}$. 
Suppose that we are given the values of $\beta(t)$ sampled at $N+1$ equidistant points of $I$, i.e., we are given $\{\beta(t_i)\}$ for $t_i = \frac{i}{N},\, i = 0,\ldots,N$. Let $\beta(t)$ 
be the piecewise geodesic connecting the points $\beta(t_i)$.
We aim to find points $\alpha(t_i)\in\operatorname{SL}(n, \mathbb R)$ such that
\begin{enumerate}[(a)]
	\item\label{PDSM.propertyLift(a)}
	$\beta(t)=\pi(\alpha(t))$ for all $t\in I$, where $\alpha(t)$ is the generalized PL curve connecting the points $\alpha(t_i)$;
	\item\label{PDSM.propertyLift(b)} $\alpha'(t)\perp\alpha(t)\mathfrak{so}(n)$ for all $t\in I$.
\end{enumerate}
We have the following algorithm to calculate this  horizontal lift $\alpha\in AC^{\perp}(I, \operatorname{SL}(n, \mathbb{R}))$:
\begin{enumerate}[(1)]
	\item Set $\alpha(0)=\beta(0)$,
	\item Given $\alpha(t_i)$, set $\alpha(t_{i+1}) = \alpha(t_i)\sqrt{\alpha(t_i)^{-1}\beta(t_{i+1})(\alpha(t_i)^{-1}\beta(t_{i+1}))^t}$.
\end{enumerate}
It is easy to see that $\alpha$ satisfies the first condition \eqref{PDSM.propertyLift(a)}. 
By Lemma \ref{lem.PDSM.effi} we choose $\alpha(t_{i+1})$ to be the element of the orbit
$\beta(t_{i+1})\cdot\operatorname{SO}(n)$ that is closest to $\alpha(t_i)$.
Thus the geodesic between $\alpha(t_i)$ and $\alpha(t_{i+1})$ is horizontal, i.e., it is perpendicular to the orbits with respect to these two elements. Thus the discrete form of the second condition \eqref{PDSM.propertyLift(b)} holds.\\

\begin{remark}
We have now described our method to calculate the horizontal lifts of curves in $\operatorname{PDSM}_{n\times n}$. To calculate the geodesic it remains to solve the 
optimization problem \eqref{eq.disSM}. 
In the case $n=2$, i.e, for  $K= \operatorname{SO}(2)$, $\mathbb H^2 \cong \operatorname{PDSM}_{2\times 2}$,  this is a minimization over a one-dimensional compact group. Thus we can use, as an alternative to the gradient method, an evaluation method to find the optimal $y\in\operatorname{SO}(2)$ of \eqref{eq.disSM}, i.e., 
discretize the one-dimensional compact group by a finite number of points and find the optimal element of this discretization.
\end{remark}

%
%

%


\bibliography{mybibfile}

\end{document}